\documentclass[11pt,a4paper]{article}

\usepackage[a4paper,top=3.8cm,bottom=3.8cm,left=3cm,right=3cm]{geometry}%
\usepackage{hyperref}

\usepackage[T1]{fontenc}
\usepackage{amsmath,amssymb,mathtools,amsthm,amsopn}
 \usepackage{mathpazo}
\usepackage{graphicx}
\usepackage[usenames,dvipsnames]{xcolor}
  \usepackage{color}

\usepackage{tocloft}
\newcommand{\cut}{\textup{cut}}

\DeclareMathOperator{\Cut}{Cut}

\renewcommand{\Im}{\operatorname{Im}}

\usepackage{fancyhdr}
\usepackage{fourier-orns}

\usepackage{enumitem}
\usepackage{cite}

\usepackage[dayofweek]{datetime}
\usepackage{mathrsfs}
\renewcommand{\phi}{\varphi}

\newcommand{\step}[1]{\par\medskip\noindent\it#1\rm}

\newcommand{\wh}{\widehat}
\renewcommand{\H}{\mathbb{H}}

\newcommand{\e}{\varepsilon}

\newcommand{\s}{\sigma}

\newcommand{\la}{\lambda}

\newcommand{\wt}{\widetilde}

\renewcommand{\d}{\delta}

\newcommand{\p}{\partial}

\newtheoremstyle{pippo}  
  {}       
  {}       
   {\sffamily}   
 {}        
  {\bfseries}  
  {.}   
  {1ex}       
  {}           
\newtheoremstyle{pluto}  {}{}
{\slshape}  {}{\bfseries}  {.} {1ex}    {}

\newtheorem{theorem}{Theorem}[section]
\newtheorem{proposition}[theorem]{Proposition}

\newtheorem{lemma}[theorem]{Lemma}

\theoremstyle{pluto}
\newtheorem{definition}[theorem]{Definition}
\newtheorem{remark}[theorem]{Remark}
\newtheorem{example}[theorem]{Example}

\usepackage{mdwlist}

\newcommand{\R}{\mathbb{R}}
\newcommand{\G}{\mathbb{G}}
\newcommand{\F}{\mathbb{F}}
\newcommand{\N}{\mathbb{N}}

\renewcommand{\d}{\delta}
\renewcommand{\t}{\tau}

\renewcommand{\a}{\alpha}
\renewcommand{\b}{\beta}

\DeclareMathOperator{\Span}{span}

\numberwithin{equation}{section}
\numberwithin{figure}{section}
\let\oldbibliography\thebibliography
\renewcommand{\thebibliography}[1]{%
  \oldbibliography{#1}%
  \setlength{\itemsep}{0pt}%
}

\usepackage{titlesec}
\titleformat{\section}{%
\normalfont\large\bfseries}{\thesection.}{1em}{}
\titleformat{\subsection}{%
\normalfont\normalsize\bfseries}{\thesubsection.}{1em}{}

\begin{document}

\title{Sub-Riemannian cut time and cut locus in Reiter-Heisenberg groups
 }
\author{Annamaria Montanari  \thanks{Dipartimento di Matematica, Alma Mater Studiorum Universit\`a di Bologna, Italy. Email address \texttt{annamaria.montanari@unibo.it}}
\and Daniele Morbidelli \thanks{Dipartimento di Matematica, Alma Mater Studiorum Universit\`a di Bologna, Italy. Email address  \texttt{daniele.morbidelli@unibo.it}}
}

\date{}

\maketitle

 \tableofcontents

 \begin{abstract}
We study the sub-Riemannian cut time and cut locus  of a given point in a class of step-2  Carnot groups of Reiter-Heisenberg type. Following the Hamiltonian point of view, we  write and analyze extremal curves, getting  the cut time of any   of them,      and a precise description of the set of cut points.
 \end{abstract}

\section{Introduction}
In this paper we study the sub-Riemannian cut time and cut locus (of the origin) in a class of step-2  Carnot groups of Reiter-Heisenberg type. Following a Hamiltonian point of view, we will write extremal curves, and we will identify cut points as either points reached by two different minimizing geodesics, or conjugate points.  In spite of the apparent similarity   with the  well known  Heisenberg group,
we will see that these more general models display several different features.

In order to state our results, let us describe briefly Reiter-Heisenberg groups.
Let $V_1=\R^{q\times p}\times \R^{p\times 1}$ and $V_2=\R^{q\times 1}$.  On $ V_1\times V_2$ define the operation
\begin{equation}\begin{aligned}
\label{reiter}
(x,y, t)\cdot(\xi, \eta, \t): =& \Big(x+\xi, y+\eta, t+\t+Q((x,y), (\xi, \eta))\Big)
\\ : = &\Big(x+\xi, y+\eta, t+\t+\frac 12(x\eta-\xi y)\Big).
\end{aligned}\end{equation}
It turns out that $\G_{qp}=(V_1\times V_2, \cdot)$ is a step-2 nilpotent  Carnot group (see~\cite{BonfiglioliLanconelliUguzzoni}, \cite{AgrachevBarilariBoscain}).   Following the terminology in   \cite{Reiter74,TorresLopera88} and \cite{Martini15}, we call it a Reiter-Heisenberg group. These models are a significant generalization of the familiar Heisenberg group which involves nontrivial interesting issues. For example, we will see that for any $q\geq 2$, the group $\G_{qp}$ contains abnormal geodesics, while for $q=1$, $\G_{1p}$ is the standard Heisenberg group,   where  minimizers are always strictly normal.

To equip $(\G_{qp}, \cdot)$ with a sub-Riemannian structure, we introduce on $V_1$ the inner product
  $\langle(x,y), (\xi, \eta) \rangle:=
\operatorname{trace}x^T\xi+ \operatorname{trace}y^T\eta =\operatorname{trace}x^T\xi+   y^T  \eta $.
    On matrix spaces, we shall always use the Hilbert-Schmidt inner product $\langle a, b\rangle :=\operatorname{trace} a^Tb$ for all $a,b\in\R^{\mu\times\nu}$ and for all $\mu,\nu\in\N$.

In order to write length-minimizing curves, we will adopt  the Hamiltonian point of view (see~\cite{AgrachevBarilariBoscain}, see Section~\ref{curvette}).  Denote by $H((x,y,t), (\xi, \eta, \t))$ the quadratic Hamiltonian for normal extremals (which will be written in Section~\ref{curvette}). Extremal curves are  smooth and are  parametrized by $(\xi, \eta, \t)\in T^*_0(V_1\times V_2)\simeq V_1\times V_2$. We denote them by $\gamma(\cdot, \xi, \eta, \t):\R\to V_1\times V_2$. We always assume $\gamma(0,\xi,\eta,\t)=(0,0,0)$
(different starting points can be easily managed by group translations, see~\cite{BonfiglioliLanconelliUguzzoni}).  The (constant) horizontal speed of such path is by definition $|(\xi,\eta)|=\sqrt{|\xi|^2+|\eta|^2}$.
We denote by $t_{\rm{cut}}\in\left]0, +\infty\right]$ the related cut time
\begin{equation*}
 t_{\rm{cut}}=t_{\rm{cut}}(\xi, \eta, \t):=\sup\{\bar s>    0:\gamma(\cdot,\xi, \eta, \t) \text{ minimizes length on $[0, \bar s]$ }
 \}.
\end{equation*}

To state our first result, for    $\t\in\R^q\setminus\{0\}$, introduce the notation
  $P_\t=\frac{\t\t^T}{|\t|^2 }\in\R^{q\times q}$ to denote
  the orthogonal projection on $\Span\{\t\}\subset\R^q$.
Let also  $P_\t^\perp: = I_q-P_\t $  be the orthogonal projection along $\t^\perp$.
On a matrix $x=[x_1, \dots, x_p]\in\R^{q\times p}$,   where
$x_k\in\R^{q\times 1}\simeq\R^q$ for $k=1,\dots, p$ are the columns of~$x$, the operators   $P_\t x$  and   $P_\t^\perp x$  project separately each column of $x$.

  Then we have the following theorem.
 \begin{theorem}[Cut-time]\label{enuncia}
Consider $(\xi, \eta)\in V_1\setminus \{0\}$ and $\t\in\R^q\setminus \{0\}$. Assume also that  $|\eta|+|P_\t\xi|>0$. Then, the length-extremal $s\mapsto \gamma (s, \xi, \eta, \t)$ such that $\gamma(0, \xi, \eta, \t)=(0,0,0)$ minimizes length up to
\begin{equation}\label{azione}
 t_{\rm cut}(\xi,\eta,\t)=\frac{2\pi}{|\t|}.
\end{equation}
Cut points of the origin can be consequently described as points of the form $\gamma(\frac{2\pi}{|\t|}, \xi, \eta, \t)$ with $(\xi, \eta, \t)$ satisfying the previous requirements.
 \end{theorem}
If one of the assumptions $\t\neq 0$ and $|\eta|+|P_\t\xi|>0$ is violated, then we have an extremal of the form $\gamma(s,\xi, \eta, 0)=(s\xi, s\eta, 0)$ which minimizes length globally (see Remark~\ref{notflat}).
Some of these lines are normal minimizers, some other are abnormal  (abnormal extremal curves will be identified
in~Proposition~\ref{sedia}).

To comment on Theorem~\ref{enuncia}, let us look at the ``horizontal speed vector'' $    (\dot x, \dot y)$ of a given extremal $\gamma(\cdot, \xi, \eta, \t)$.  We have
$,
 (\dot x(s), \dot y(s))= a(\xi, \eta)\cos(|\t|s)+ b(\xi, \eta)\sin (|\t|s)+z(\xi, \eta)
$, for suitable functions $a,b,z$ of the variables $\xi,\eta$ (see Section~\ref{preliminari}).
The cut time corresponds to a complete period  of the circular functions. This agrees with the standard Heisenberg group and with Heisenberg-type groups~\cite{AutenriedMolina16}.
There is however a constant part $z(\xi,\t)$ which, if nonzero, makes the ``horizontal part'' $(x,y)$ of a cut point nonzero.   Previous results with similar features on cut times for Carnot groups of step two were proved   in \cite{AgrachevBarilariBoscain12,BarilariBoscainGauthier12} for structures of low corank (at most two). The case of  free, step-2 Carnot groups is dealt in  \cite{Myasnichenko02,MontanariMorbidelli17,RizziSerres16}. Here, working in model~\eqref{reiter},  we have no bounds on the corank $\dim(V_2)$. The comprehensive survey paper~\cite{Sachkov22} should be consulted to have a complete account on the mentioned models and on sub-Riemannian manifolds outside the setting of Carnot groups.  In the recent papers \cite{Li} and \cite{LiZhang}, the authors  analyze the cut-time in a rather large class of step-two Carnot groups, including the limiting case $\G_{q1}$ of the  family of models $\G_{qp}$ object  of the present paper.

Our second result involves the description of the cut locus as a set.  It turns out
from Theorem~\ref{enuncia}
and from the form of extremals written in Section~\ref{dasopra},    that cut points have the form
 \begin{equation*}
\left\{
\begin{aligned}
 x\Big(\frac{2\pi}{|\t|},\xi,\eta,\t\Big)& =
 \frac{2\pi}{|\t|} P_{\tau}^\perp \xi
\\ y\Big(\frac{2\pi}{|\t|},\xi,\eta,\t\Big) &=0
\\
t\Big(\frac{2\pi}{|\t|},\xi,\eta,\t\Big)& =
\frac{\pi}{|\t|^2}
\Big(|P_\t\xi|^2
+|\eta|^2\Big)\frac{\t}{|\t|}
 -\frac{2\pi}{|\t|^2}  P_\t^\perp \xi \frac{\xi^T\t}{|\t|}
\end{aligned}
\right.
 \end{equation*}
where $\t\neq 0$ and  $|\eta|^2+|P_\t\xi|^2>0$.   Note first that if $q=1$ and $\t\in\R^1\setminus\{0\}$, trivially we have $\t^\perp=\{0\}$ and $P_\t^\perp=0$. Then cut points have the form $(0,0,t)$, i.e.~they are contained in the $t$-axis. This is the familiar case of the Heisenberg group $\H^p$ (See \cite[Chap.~13]{AgrachevBarilariBoscain}). If instead $q>1$, then $P_\t^\perp$ is always nonzero and we have points $(x,0,t)\in \Cut(\G_{qp})$ with $x\neq 0$.
 The natural question is now whether or not we may have $\Cut(\G_{qp})=\{(x,0,t)\in V_1\times V_2\}$, for some choices of $q$ and $p$.
The answer is not, by the following theorem.   To state it, given $x\in\R^{q\times p}$, introduce
 $P_{\Im x} $ and   $P_{\Im x}^\perp :\R^q\to \R^q $ to
 denote respectively  the orthogonal projection on $\Im x:=\{xy:y\in\R^p\}\subset\R^q$ and on its orthogonal.
 The linear map $x^\dag\in \R^{p\times q}$ denotes instead the Moore-Penrose inverse of $x$ (see below).
\begin{theorem}[Identification of the cut locus]\label{cutolo}
Let $q,p\in\N$ and let $\G_{qp}$ be the associated Carnot group.
 Then
\begin{enumerate}[nosep]
 \item \label{unno} We have
 \begin{equation}\label{etichetta}
\begin{aligned}
  \Cut(\G_{qp})=\big\{(x,0,t)\in\G_{qp} :  t\notin \Im x \; \text{ and }\;
  |P_{\Im x}^\perp t|\geq \pi |x^\dag t|^2\big\}.
  \end{aligned}
\end{equation}
\item \label{duee}    If $p\geq 2$, then all cut points are conjugate points.  \item \label{terzi}  A cut point $(x,0,t) $ is reached by a unique unit-speed length-minimizing curve if and only if   equality $  |P_{\Im x}^\perp t|= \pi |x^\dag t|^2 $ holds.
\item \label{tres} Finally, the distance from the origin of a cut point $(x,0,t)$ has the form
\begin{equation}\label{pece}
 d(x,0,t)  =\sqrt{|x|^2+4\pi|P_{\Im x}^\perp t|}.
\end{equation}
\end{enumerate}
\end{theorem}
 If $p=1$, item \ref{duee} is drastically different. Note that for $p=1$ we have $x^\dag t=\frac{\langle x,t\rangle}{|x|^2}$.

\begin{theorem}\label{loggia}
If $p=1$ and $q\geq 2$, a cut point $(x,0,t)$ is conjugate if and only if
\begin{equation}\label{watts}
|x| \big( |P_{\Im x}^\perp t|-\pi |x^\dag t|^2\big)=0.
\end{equation}
\end{theorem}
In the statement above,
for the notion of conjugate point, see Definition~\ref{pcon}.   The conjugate   point $(x,0,t)$ is necessarily a \emph{first conjugate point}, see   Remark~\ref{capperi}. \color{black}
   For the Moore-Penrose inverse   $x^\dag\in \R^{p\times q}$  of $x$,  see e.g.~\cite{Meyer-matrix-analysis}. Precisely, given $t\in\R^q$,   $x^\dag t$    is   uniquely defined by $x x^\dag t=P_{\Im x} t$ and
$x^\dag t\perp \ker x$. If $\ker x=\{0\}$, then  $x^\dag t$ is the unique solution~$y\in\R^p$ of the system $xy=P_{\Im x}t$. Otherwise, $x^\dag t$ is the smallest-norm element of the set of solutions  $\{y= x^\dag t+\eta:\eta\in\ker x\}$ of  $x y=P_{\Im x}t$. In other words,
$|x^\dag t |=|x^\dag P_{\Im x}t|=\min\{|y|:y\in \R^p\text{ and } xy =P_{\Im x} t\}$.
 Note that we do not include  $(0,0,0)$  in our definition of  cut-locus.
 Observe that~\eqref{etichetta} implies that
$\Im x\neq \R^q$ for any point $(x,0,t)\in\Cut(\G_{qp})$.
  Furthermore, since given $x=0\in\R^{q\times p}$, we have  $0^\dag =0\in\R^{p\times q}$, all points of the form $(0,0,t)$   with $t\neq 0$   belong to the cut locus as expected.   Finally,    for all $r>0$ the set in~\eqref{etichetta} is dilation invariant with respect to the standard Carnot homogeneous dilation $(x,y,t)\mapsto(rx, ry, r^2 t) $.

\begin{remark}  Some remarks  on Theorems~\ref{enuncia} and~\ref{cutolo}
are now in order.
 \begin{enumerate}[nosep]
%

\item In Theorem~\ref{cutolo}, we describe precisely all cut points  reached by a unique length-minimizer.  Such kind of minimizers are absent in  $p$-dimensional Heisenberg groups $\G_{1p}$ (case  $(q=1)$), and in the free step-2, rank-three model (not included in our class~$\G_{qp}$). They appear, but they remained unnoticed  in~\cite{BarilariBoscainGauthier12}. Points of such kind appear classically on Riemannian equatorial geodesics in oblate revolution ellipsoids (See~\cite{Klingenberg}).   In Remark~\ref{insiemedenso} we prove that points in $\Cut(\G_{qp})$ reached by more than one unit-speed length-minimizer are dense in $\Cut(\G_{qp})$. See~\cite{Warner65},  for the Riemannian analogous.
\color{black}

\item \label{tre}  The set  $\Cut(\G_{qp})\cup\{(0,0,0)\}$ is not closed for any $q\geq 2$ and $p\in\N$.  Indeed, consider a family $(x,0, \e t)$ with $x\neq 0$, $t\neq 0$ and $t\perp\Im x$ (this forces $q\geq 2$). We have $(x,0, \e t)\in \Cut(\G_{qp})$
for all $\e>0$, but $(x,0,0)$ is not a cut point for any $x\in\R^{q\times p}  $.
In  Proposition~\ref{sedia}, we prove  that these points    $(x,0,0)\neq 0$  belonging to  $\overline{\Cut(\G_{qp})}\setminus
\Cut(\G_{qp})$ are always abnormal points.
Recall that, in absence of abnormal minimizers, Rifford and Tr\'elat proved in~\cite[Lemma~2.11]{RiffordTrelat09} that   in a sub-Riemannian Carnot group $\G$ the set  $\Cut(\G)\cup \{0\}$ is closed.

\item At abnormal points appearing   in item~\ref{tre},
the function $t_\cut:T^*_0\G_{qp}\to \left]0,+\infty\right]$,
defined as $(\xi,\eta,\t)\mapsto t_\cut(\xi,\eta,\t)$ has a discountinuous behaviour (upper semicontinuous, unbounded), as a function of the dual variables $(\xi, \eta,\t)$. See Example~\ref{discontinua}. The function $(s,\xi,\eta,\t)\mapsto \gamma(s,\xi,\eta,\t)$ is instead smooth in all its arguments, by standard ODE theory.

\item Observe that $\Cut(\G_{qp})$ can be a quite large set. Namely, if $p=1$, it turns out that $\Cut(\G_{q1})$ contains the    set   $\{(x,0,t)\in\G_{q1}:
|P_{\Im x}^\perp t| \gneqq \pi|x^{\dag} t|^2$ and $\operatorname{rank}x $ is maximal$\}$, which is
an open\footnote{By \cite{Stewart}, the map $x\mapsto x^\dag$ is continuous on the open set $\{x\in\R^{q\times p}:\operatorname{rank}(x) $ is maximal$ \}$.}
dilation invariant subset of the codimension-one hyperplane   of equation  $y=0$.     Note that in free step-2 Carnot groups $\F_k$ of rank $k=2$ and $k=3$,   the cut-locus is a smooth  manifold   of dimension $\operatorname{dim}(\Cut(\F_k))=\dim(\F_k)-2$ (for the case $k=3$, see \cite[Remark~4.3]{MontanariMorbidelli17}).   The conjectured  dimension of $\Cut(\F_k)$ is again
$\dim(\F_k)-2$ in  rank-$k$, step-2  groups $\F_k$ with $k\geq 3$.    See \cite{RizziSerres16}.
 Observe finally that in the limiting case $p=1$  our set~\eqref{watts} agrees with the one found by~\cite[Section~10]{LiZhang} with completely different methods.

\item Finally, note that in the model $\G_{q1}$, all points $(x,0,t)\in\Cut$ such that~\eqref{watts} is violated give examples of extremals whose cut time is strictly less than the first conjugate time. See~\cite{BarilariBoscainGauthier12} for previous different examples.

\end{enumerate}
\end{remark}

Since the cut locus appearing in~\eqref{etichetta} is not easy to  visualize for general values of $q,p$, in Subsection~\ref{zoppo} we  discuss in some details  the group $\G_{q1}$. In that case, the cut locus turns out to be defined as a sublevel set of an explicit scalar polynomial function. Its regularity properties are analyzed in Proposition~\ref{ciliegia}.

 Let us describe now the structure of the paper. In Section~\ref{preliminari}, after providing general notation and known facts, we write the length-extremals of our sub-Riemannian problem and we characterize abnormal ones.
 To prove Theorem~\ref{enuncia},  starting from the candidate cut time $\frac{2\pi}{|\t|}$ appearing in~\eqref{azione}, we show in Section~\ref{dasopra} the upper estimate $t_\cut \leq \frac{2\pi}{|\t|}$ for extremals which are not Euclidean lines. This is achieved by
 the analysis of conjugate points (Propositions~\ref{massetto} and~\ref{piuno}). We also characterize points reached by a unique minimizer (see Proposition~\ref{masso}). In Section~\ref{dalbasso},
 we prove the lower bound $t_\cut\geq \frac{2\pi}{|\t|}$ by using  a   geometric-control argument    in part   inspired to
 the paper~\cite{BarilariBoscainGauthier12}. Finally, in Section~\ref{cinque}, we conclude the  proof of Theorem~\ref{cutolo}.
In Subsection~\ref{zoppo}, we also explicitly  describe $\Cut(\G_{q1})$ and we analyze its regularity.
  In Section~\ref{cinese}, we prove Theorem~\ref{loggia}.

\section{Reiter-Heisenberg groups and  their length-extremals}\label{preliminari}
In this section we briefly recall the notion of sub-Riemannian length and distance in Reiter-Heisenberg groups. Then we write the explicit form of extremals and,   among them, we characterize the abnormal ones.

\subsection{General facts}
Let $(\G_{qp}, \cdot)$ be the Reiter-Heisenberg group defined in~\eqref{reiter}. A horizontal curve is a  Lipschitz-continuous solution $\gamma=(x,y,t):[0, T]\to V_1\times V_2$     of the   ODE
\begin{equation}\label{senza}
 (\dot x,\dot y)=(u, v) \quad \text{ and } \dot t = Q((x,y), (u, v)),\quad \text{a.e. on $[0,T]$, }
\end{equation}
where $(u,v): [0,T]\to V_1  $ is   an $L^{\infty}$ control.   The horizontal speed of $\gamma$ is $|\dot\gamma |_{\text{hor}}:=\sqrt{|u|^2+|v|^2}:=\sqrt{(\operatorname{trace}u^T u)^2 +|v|^2}$. The length of $\gamma $ is $\int_0^T  |\dot\gamma |_{\text{hor}}(s) ds
$. Since we have  H\"ormander's rank condition $\Span \{Q((x,y), (\xi, \eta)):  (x,y), (\xi, \eta)\in V_1\}=V_2$, it turns out that any pairs of points can be connected by a horizontal curve   (this follows from Chow--Rashevskii theorem).   Minimizing such length we obtain  the  well known sub-Riemannian distance.

Let us introduce some notation in $\G_{qp}$. We sometimes identify $\R^p$ with $\R^{p\times 1}$ and the same for $\R^q$.
As we declared, we use the Hilbert-Schmidt  inner product in $\R^{q\times p}$, i.e.~ $\langle x, \xi\rangle:=\operatorname{trace}( \xi^T x)$.  Then, in $\R^{q\times p}\times\R^p$ we define $\langle(x,y), (\xi, \eta)\rangle :=\langle x,\xi\rangle +\langle y,\eta\rangle$, so that $e_\a e_k^T$ as $\a\in\{1,\dots, q\}$ and $k\in\{1,\dots,p\}$ is an orthonormal basis of~$\R^{q\times p}$,    where $e_\a$ and   $e_j$, with  $\a\in\{1,\dots, q\}$ and  $j\in\{1,\dots p\}$   denote the canonical basis of $\R^q$ and $\R^p$.
Introduce,
 for $\t\in\R^q$, the linear  map    $ A_\t:V_1\to V_1$
\begin{equation*}
 A_\t
(\xi,\eta)
=(  \t \eta^T,
 -\xi^T\t).
\end{equation*}
Since the map is also linear in $\t$, we have   $A_\t =\sum_{\a=1}^q\t_\a A_{e_\a}$, with  $A_{e_\a}( \xi,  \eta) =
( e_\a \eta^T , -\xi^T e_\a)$
 for all $(\xi,\eta)\in V_1$.
Thus, we have
\begin{equation*}\begin{aligned}
 \langle
( x,  y) ,  A_{e_\a} (  \xi,   \eta)
\rangle & =
 \langle (x,y),
( e_\alpha \eta^T , -\xi^T e_\a)\rangle
=\langle x, e_\alpha \eta^T\rangle- \langle y, \xi^T e_\a\rangle =\langle x\eta-\xi y ,e_\a\rangle.
\end{aligned}
\end{equation*}
Each map $A_{e_\a}$ is skew-symmetric and we have $  Q((x,y),(\xi,\eta))=\frac 12\sum_{\a=1}^q \langle
( x , y) , A_{e_\a} ( \xi,  \eta)
\rangle e_\a$.

\subsection{Extremal curves}\label{curvette}
We are interested in writing length-minimizing curves.
In order to write   them,    we follow \cite[Section 13.1]{AgrachevBarilariBoscain}.
Note that in step-2 Carnot groups it is known that all extremals are normal (see  \cite[Corollary
12.14]{AgrachevBarilariBoscain} ).    We first write a frame  of left-invariant horizontal orthonormal vector fields. For
$\a\in\{1,\dots,q\}$ and $k,j\in\{1,\dots, p\}$ we have
\begin{equation}\label{campetti} \begin{aligned}
 X_{\a k}(x,y,t)& =\frac{d}{ds}\Big|_{s=0}(x,y,t)\cdot(s e_{\a} e_k^T, 0, 0)=\Big(e_{\a  }e_k^T, 0, -\frac 12 y_k e_\a\Big)
 \\
 Y_j(x,y,t)& =\frac{d}{ds}\Big|_{s=0}(x,y,t)\cdot (0, se_j, 0)=\Big(0, e_j, \frac 12 x e_j)=\Big(0, e_j, \frac 12 x_j),
\end{aligned}
\end{equation}
where we wrote $x=[x_1,\dots, x_p]$ with $x_j\in \R^q$ for all $j\in\{1,\dots, p\}$. Introducing the functions  $u_{\a k}(x,y,t,\xi,\eta,\t):= \langle (\xi,\eta, \t),  X_{\a k}(x,y,t)\rangle$ and $v_j
(x,y,t,\xi,\eta,\t)=\langle (\xi,\eta, \t),  Y_j(x,y,t)\rangle$,  extremals are furnished by the Hamiltonian
\begin{equation*}
\begin{aligned}
 H((x,y,t), (\xi, \eta, \tau))& =
 \frac 12\sum_{\a,k}
 u_{\a k}(x,y,t,\xi,\eta,\t)^2+ \frac 12\sum_{j}v_j(x,y,t,\xi,\eta,\t)^2.
 \end{aligned}
\end{equation*}
Namely, to obtain all minimizers from $(0,0,0)$, one integrates the Hamiltonian system $(\dot  x, \dot  y,
\dot t)=\nabla_{(\xi, \eta, \t)}H$ and $(\dot\xi,\dot \eta, \dot \t)=-\nabla_{(x,y,t)}H$ with
initial data $(x(0), y(0), t(0))=(0,0,0)$ and $(\xi(0), \eta(0), \t(0))=(\xi, \eta, \t)\in T^*_0\G_{qp}\simeq V_1\times V_2$. \footnote{  Since we are taking global coordinates $((x,y,t), (\xi,\eta,\tau))\in (V_1\times V_2)\times (V_1\times V_2)$ on $T^*(V_1\times V_2)$, we identify covectors in $T^*_{(0,0,0)}(V_1\times V_2)$ with $(\xi,\eta,\tau)\in V_1\times V_2$.} It
turns out that extremals from the origin are horizontal curves
\begin{equation}\label{laga}
s\mapsto \gamma(s)=\gamma(s, \xi, \eta, \t)\in V_1\times V_2                                                                                                                           \end{equation}  parametrized by $(\xi, \eta, \t)\in V_1\times V_2$.
Furthermore, again, by \cite[Section 13.1]{AgrachevBarilariBoscain},
given $(\xi,\eta,\t)$, the  extremal curve
$\gamma(\cdot, \xi, \eta, \t)$ is the solution of~\eqref{senza}, with
\begin{equation}\label{tremo}
( u(s),v(s))
= e^{-s A_\t}(\xi,  \eta)\in V_1\times V_2.
\end{equation}

Next we give the form of  extremal curves in terms of the three functions
\begin{equation*}
T(\phi)=\frac{\sin\phi}{\phi},\quad  U(\phi) =\frac{\phi-\sin\phi\cos\phi}{4\phi^2}\quad\text{and}\quad V(\phi)=\frac{\sin\phi-\phi\cos \phi} {2\phi^2}
\end{equation*}
defined for $\phi>0$.
\begin{proposition}\label{monte} For all $(\xi, \eta,\t)\in T^*(V_1\times V_2)\simeq V_1\times V_2$ with $\t\neq 0$, the curve $\gamma(\cdot,\xi, \eta, \t)=(x(\cdot,\xi, \eta, \t),y(\cdot,\xi, \eta, \t),t(\cdot,\xi, \eta, \t))$ in~\eqref{laga} has the form
\begin{equation}  \begin{aligned}\label{pari}
 x(s,\xi, \eta, \t) &= s T\big(\tfrac{|\t|s}{2}\big)
 \Big\{
 P_\t\xi\,
\cos\big(\tfrac{|\t|s}{2}\big)
-\frac{ \t\eta^T}{|\t|} \sin\big(\tfrac{|\t|s}{2}\big)
\Big\}
  + s
P_{\tau}^\perp \xi
\\
y (s,\xi, \eta, \t) &=
s T\big(\tfrac{|\t|s}{2} \big)\Big\{\eta \cos \big(\tfrac{|\t|s}{2} \big)+
\frac{\xi^T\t}{|\t|}
\sin\big(\tfrac{|\t|s}{2} \big)\Big\}
\\ t
(s,\xi, \eta, \t) &= s^2 U\big(\tfrac{|\t|s}{2} \big)
\big\{|P_\t\xi|^2 +|\eta^2| \big\}\frac{\t}{|\t|}
\\& +s^2 V\big(\tfrac{|\t|s}{2} \big)
P_\t^\perp\xi  \Big\{-  \eta  \sin\big(\tfrac{|\t|s}{2}\big)
+ \frac{\xi^T\t}{|\t|}   \;
\cos \big(\tfrac{|\t|s}{2} \big)
\Big\}.
\end{aligned}\end{equation}
\end{proposition}
In formula~\eqref{pari},  recall that $P_\t:=\frac{\t\t^T }{|\t|^2}\in\R^{q\times q}$ and $P_\t^\perp =I_q-P_\t$. Note also that $|P_\t\xi|^2=\frac{|\xi^T\t|^2}{|\t|^2}$. Observe  the   known  property
\begin{equation}\label{benny}
 \gamma(\la s,\xi,\eta,\t)=\gamma(s, \la\xi,\la\eta,\la\t),\quad\text{for all $\la>0$, $(\xi,\eta,\t)\in V_1\times V_2$.}
\end{equation}

\begin{proof} Since
\begin{equation}\label{matera}
 A_\t^2
(      \xi,\eta)=(   -\t\t^T\xi , -|\t|^2\eta)
 \text{ and }
A_\t^3
(\xi, \eta)=  -|\t|^2(\t  \eta^T, -\xi^T\t) =-|\t|^2 A_\t(\xi, \eta),
\end{equation}
summing up the series we get
\begin{equation*}\begin{aligned}
(  u(s),  v(s))&= ( \xi ,\eta)
  -\frac{\sin(|\t|s)}{|\t|}
 (\t  \eta^T ,-\xi^T\t
)
 -\frac{1-\cos(|\t|s)}{|\t|^2}
(  \t\t^T\xi  , |\t|^2\eta).
\end{aligned}
\end{equation*}
Let now
\begin{equation*}
 a=a(\xi, \eta):=(P_\t\xi, \eta),\quad b=b(\xi, \eta):=
 \Big(  -\frac{\t \eta^T}{|\t|}
 , \frac{\xi^T\t}{|\t|} \Big),\quad z=z(\xi, \eta):=(P_\t^\perp\xi,0).
\end{equation*}
  Then, we can write
$
 (u(s), v(s)) = a\cos(|\t|s) + b\sin(|\t|s) +z
$.
Using the function $T(\phi):=\frac{\sin\phi}{\phi}$  and by trigonometry we get
\begin{equation*}\begin{aligned}
(  x(s), y(s))
&=\frac{\sin(|\t|s)}{|\t|}a+ \frac{1-\cos(|\t|s)}{|\t|}b + sz
\\
=& s T\Big(\frac{|\t|s}{2}\Big) \bigg\{
 a  \cos\Big(\frac{|\t|s}{2}\Big) +
b \sin\Big(\frac{|\t|s}{2}\Big)
                       \bigg\}
  +sz
\\=& s T\Big(\frac{|\t|s}{2}\Big) \bigg\{
                           (P_\t \xi ,   \eta  )
   \cos\Big(\frac{|\t|s}{2}\Big) +
\Big( -\frac{\t\eta^T}{|\t|}  , \frac{\xi^T \t}{|\t|}\Big)
 \sin\Big(\frac{|\t|s}{2}\Big)
                       \bigg\}
 +s (
P_\t^\perp\xi ,  0 ).
 \end{aligned}
\end{equation*}
To calculate $t(s)=\int_0^sQ((x , y ), (u  , v ))  $, integrating and by bilinearity we get
\begin{equation}\label{duse}
 \begin{aligned}
t(s)= &
\int_0^s Q\bigg( \frac{\sin(|\t|\s)}{|\t|}a + \, \frac{1-\cos(|\t|\s)}{|\t|}b+\s z,
 a\cos(|\t|\s)+b\sin(|\t|\s)+z\bigg)d\s
\\&=
\frac{|\t| s-\sin(|\t| s)} {|\t|^2}
Q(a,  b)
+
\frac{2(1-\cos(|\t| s))  - |\t| s\sin(|\t| s)}{|\t|^2}
Q(a, z)
\\& \qquad \qquad \qquad \qquad \qquad+
\frac{|\t| s(1+\cos(|\t| s)) - 2\sin(|\t| s)}{ |\t|^2}
Q(b,  z).
 \end{aligned}
\end{equation}
The form of $t$ has a structure analogous to \cite[eq.~(2.6)]{MontanariMorbidelli17}.
  By the definition   $Q((x,y), (\bar x, \bar y)):=\frac 12(x\bar y-\bar x y)$ and   $\frac{|\xi^T\t|^2}{|\t|^2}
=|P_\t\xi|^2 $,  an easy calculation gives
\begin{equation*}
\begin{aligned}
   Q(a,b)  &=\frac 12\Big\{\frac{|\xi^T\t|^2}{|\t|^2}+|\eta|^2\Big\}\frac{\t}{|\t|}
=\frac 12 (    |P_\t \xi |^2  +|\eta|^2 ) \frac{\t}{|\t|}
\\  Q(a,z) &  =
  -\frac 12P_\tau^\perp \xi\,\eta \in \R^{q\times 1}
\quad\text{and}\quad  Q(b,z)  =  -\frac 12
P_{\t}^\perp\xi  \frac{ \xi^T\t}{|\t|}.
\end{aligned}
\end{equation*} Therefore, starting from~\eqref{duse},   we obtain
\begin{equation}\label{ti}
\begin{aligned}
 t(s)&=\frac{|\t|s-\sin(|\t|s)}{2|\t|^2} ( |P_\t\xi|^2+|\eta|^2 )\frac{\t}{|\t|}
-\frac{2(1-\cos(|\t|s))-|\t|s\sin(|\t|s)}{2|\t|^2}
P_\t^\perp\xi\,\eta
 \\&\qquad
 -\frac{|\t|s(1+\cos(|\t|s))-2\sin(|\t|s)}{2|\t|^2} P_\t^\perp\xi  \frac{\xi^T\t}{|\t|}.
\end{aligned}
\end{equation}
Writing trigonometric functions in terms of  $ \frac{|\t|s}{2} $ instead of $|\t|s$, we get
\begin{equation}\label{est}
\begin{aligned}
t(s, \xi, \eta, \t) & = s^2U(  \tfrac{|\t|s}{2} ) (
|P_\t\xi|^2+|\eta|^2
 )\frac{\t}{|\t|}
\\&\qquad +s^2V   ( \tfrac{|\t| s}{2} )
\Big(-\sin ( \tfrac{|\t| s}{2} ) P_{\tau}^\perp \xi\,\eta  +\cos(   \tfrac{|\t| s}{2} )P_{\t}^\perp
\xi \frac{ \xi^T\t}{|\t|}  \Big),\end{aligned}
\end{equation}
as desired. The proof is finished.
\end{proof}

\begin{remark}\label{notflat}
 It is  easy to check that constant-speed  Euclidean lines of the form $(x(s), y(s), t(s))= (su, sv, 0)$  for some $(u,v)\in V_1\setminus\{(0,0)\}$ are always globally minimizing. Furthermore, an extremal of the form  $(u(s), v(s))=e^{-s A_\t}(\xi, \eta)  $  gives rise to  an Euclidean line of that form if and only if
 \begin{equation*}
  \left\{\begin{aligned}
 &\t=0\\&|\xi|^2+|\eta|^2>0 \end{aligned}\right.\qquad\text{or}\quad \left\{\begin{aligned}
& |\t| \,   |P_\t^\perp\xi |>0
\\& |\eta|^2+ |P_\t\xi |^2=0.
\end{aligned}
\right.
 \end{equation*}
In the first case, $\dot\gamma(s)=(\xi,\eta,0)$ and we have $\gamma(s)=(s\xi,s\eta,0)$. In the second case we have instead $\xi^T\t=0$ and~\eqref{pari} implies that $\gamma(s)=(sP_\t^\perp\xi, 0,0)=(s[P_\t^\perp\xi_1,
\dots, P_\t^\perp\xi_p],0,0)$. Note that in both cases we have $A_\t(\xi, \eta)=(\t\eta^T, -\xi^T\t)=0$,   getting     $e^{-s A_\t}(\xi, \eta)=(\xi,\eta) $ trivially.
\end{remark}

We conclude this section with a characterization of abnormal extremal curves.
Recall that, by definition,
given a~$L^2$  control $(u, v): (0, T)\to V_1$, the curve $\gamma_{(u, v)}:[0,T]\to \G_{qp}$ is abnormal if $d_{(u, v)}E:L^2((0, T)
\to \G\sim T_{E(u,v)}\G)$ is not onto. Here $E:L^2(0,T)\to \G_{qp}$ is the end point map, i.e. $E(u,v):=\gamma_{(u,v)}(T)$, where $\gamma_{(u,v)}$ is the curve corresponding to $(u,v)$ and with $\gamma(0)=(0,0,0)$.   We say instead that $\gamma_{(u,v)}$ is \emph{strictly normal} when it is not abnormal in any  subsegment of $[0,T]$.

\begin{proposition}\label{sedia}
  Let $q\geq 2$ and let     $(\xi, \eta, \t)
\in \G_{qp}\simeq T^*_{(0,0,0)}\G_{qp}$ and assume that   $|\xi|^2+|\eta|^2>0$. The curve $\gamma (\cdot,\xi, \eta, \t)$ is abnormal if and only if it has the form $\gamma(s)=(s\xi, 0, 0)$ with $\Im\xi\neq\R^q$.
\end{proposition}
  If $q=1$, all extremals are strictly normal. Recall also    that it is known from second-order analysis that in step-2 Carnot groups all abnormal curves are also normal. See \cite[Corollary~12.14]{AgrachevBarilariBoscain} and  \cite[Section~20.5]{AgrachevSachkov04}.

\begin{proof}
By
\cite[Sect.~3]{MontanariMorbidelli15}, we know that an extremal curve $\gamma_{(\xi, \eta, \t)} =(x_{(\xi, \eta, \t)},y_{(\xi, \eta, \t)},t_{(\xi, \eta, \t)})$ corresponding to a given control
$(u(s),v(s))=e^{-s A_\t}(\xi ,\eta)$  is abnormal if and only of there is $\sigma\in\R^q\setminus\{0\}$ such that
\begin{equation}\label{abno}
A_\s e^{-s A_\t}(\xi, \eta)=(0,0)\quad\text{ for all $s\in\R$.} \end{equation}

In order to show the ``if'' part, consider $(u(s), v(s))=(\xi,0)$ for all $s\in[0,T]$,  where $\Im \xi\neq \R^q$.   Take $\sigma\in (\Im\xi)^\perp \setminus\{0\}$. We have $A_\s(\xi,0)=(0,-\xi^T\sigma)=(0,0)$ and~\eqref{abno} follows.

 To show the ``only if'' part,
let us look first at case $\t=0$. In this case $(u(s), v(s))=(\xi, \eta)$ and condition~\eqref{abno} furnishes $A_\s (\xi, \eta)=(\s \eta^T, -\xi^T\s)=(0,0)$. Since $\sigma\in\R^q $ must be nontrivial, first coordinate gives $\eta=0$. The second is equivalent to $P_\s\xi=0$.
Thus, it must be $P_\s^\perp\xi\neq 0$. Existence of such a $\sigma\neq 0$ is equivalent to condition $\Im\xi\neq \R^q$.

Let finally $(u(s), v(s))=e^{-sA_\t  }(\xi,\eta)$ be abnormal  with $\t\neq 0$. From the discussion above, given $\s\in\R^q\setminus\{0\}$, we have $\ker A_\s=\{(\xi, 0): \Im \xi\subset\sigma^\perp\}$.
Evaluating~\eqref{abno} and its $s$-derivative at $s=0$, we get
  $
 A_\s(\xi, \eta)=(0,0)$ and $A_\s A_\t(\xi, \eta)=A_\s(\t\eta^T, -\xi^T\t)= (0,0)$.
Therefore, we get
\[
 \eta=0,\quad \Im\xi\subset\sigma^\perp,\quad \xi^T\t=0,\quad
 \Im(\t\eta^T)\subset\sigma^\perp .
\]
Requirements $\eta=0$ and $\xi^T\t=0$ already imply that  $A_\t(\xi, 0)=(0,0)$ and $\Im\xi\subset\t^\perp$, which is a nontrivial subspace. Therefore $e^{-sA_\t}(\xi, 0)=(\xi,0)$ for all $s$, and  $\Im\xi\neq\R^q$, as we wished.
\end{proof}

  We briefly recall the notion of conjugate point.

\begin{definition}\label{pcon} Recall that, given an extremal $\gamma(\cdot, \bar\xi, \bar\eta, \bar\t)$ with $|(\bar\xi, \bar\eta)| >0$  and a time $\bar s>0$, we say that $\bar s$ is a conjugate time for $\gamma(\cdot, \bar\xi,\bar\eta, \bar \t)$ if
the differential of the map $( \xi, \eta, \t) \in\G_{qp} \mapsto \gamma(\bar s,\xi, \eta, \t)\in\G_{qp}$
is singular at $( \bar\xi, \bar\eta, \bar \t)$. \end{definition}

\begin{remark}\label{capperi} If $\gamma(\cdot,\xi,\eta,\t)$ is abnormal, then all times  $\bar s>0 $ are trivially conjugate times  (see \cite[Remark 8.46]{AgrachevBarilariBoscain}). If instead $\gamma(\cdot ,\xi,\eta, \t)$ is not abnormal in any subsegment $[0,s]$, it is known that there is a strictly positive  smallest conjugate time $  t_{\textup{conj}}$. Furthermore, we have $  t_{\textup{conj}}\geq t_{\textup{cut}}$. The  time $t_{\textup{conj}}$ is usually called \emph{first conjugate time}. The corresponding point $\gamma(t_{\emph{conj}},\bar\xi,\bar\eta,\bar\t)$ is called  \emph{first conjugate point}.  All these facts are proved in~\cite[Section 8.8]{AgrachevBarilariBoscain}.
\end{remark}


\section{Upper estimate  \texorpdfstring{$t_{\cut}\leq 2\pi/|\t|$}{upper}}\label{dasopra}
In this section we consider extremal curves $\gamma(\cdot,\xi,\eta,\tau)$ which are not Euclidean lines (in particular strictly normal, see Proposition~\ref{sedia}  and Remark~\ref{capperi}). We show  that  if $p\geq 2 $, then  all points of the form $\gamma(\tfrac{2\pi}{|\t|}, \xi , \eta, \t)$ with $\t\neq 0$ and $|\eta|+|P_\t\xi|>0$ are conjugate points. See   Definition~\ref{pcon}.   Among them, we characterize those  that are   reached by at least two different (unit-speed) minimizing
geodesics exiting from the origin (see Proposition~\ref{masso}). The remaining points are reached by a unique unit-speed minimizing curve.\
 In  $\G_{q1}$, i.~e. $p=1$ and $q\geq 2$, it is not true that all points  $\gamma(\frac{2\pi}{|\t|}, \xi, \eta, \t)$ are conjugate points.  Conjugate points in this limiting case are described in Proposition~\ref{piuno} for the sufficient part and in the separate Section~\ref{cinese}, Theorem~\ref{taglia} for the necessary part, which is slightly more technical.

Starting from the form of extremals established in Proposition~\ref{monte}, we get
\begin{equation}\label{ovest}
\begin{aligned}
 x\Big(\frac{2\pi}{|\t|},\xi,\eta,\t\Big)& =
 \frac{2\pi}{|\t|} P_\t^\perp\xi
 \\ y\Big(\frac{2\pi}{|\t|},\xi,\eta,\t\Big) &=0
\\
t\Big(\frac{2\pi}{|\t|},\xi,\eta,\t\Big)
& = \frac{\pi}{|\t|^2} ( |P_\t\xi|^2
+|\eta|^2)\frac{\t}{|\t|}
 -\frac{2\pi}{|\t|^2}  P_\t^\perp \xi\frac{\xi^T\t}{|\t|}.
\end{aligned}
\end{equation}

\begin{proposition}[Conjugate points  for $p\geq 2$]\label{massetto} Let $q\in\N$ and $p\geq 2$.
 Given $(\xi, \eta, \t)\in V_1\times V_2$, assume that $\t\neq 0$ and $|\eta|^2+|P_\t\xi|^2>0$ and consider  the extremal $\gamma(\cdot, \xi ,\eta, \t)$. Then, the time $\bar s:= \frac{2\pi}{|\t|}$ is conjugate.
\end{proposition}

  As  we already observed, if $ \t =0$ or $\t\neq 0$ and  $ |\eta|+|P_\t\xi| =0$, then $\gamma$ has the form $\gamma(s)= (s\xi,  0,0 )$.
If $\Im\xi=\R^q$, then it must be  $\t=0$, because there is no $\t\neq 0$ such that $P_\t\xi=0$. In this case,  $\gamma $ is strictly normal and it has no conjugate points.  If  $\Im\xi\neq\R^q$, then   $\gamma$ is  abnormal and all its points are trivially conjugate (see Proposition~\ref{sedia}).
\begin{proof}
Let $(\bar \xi, \bar \eta,\bar \t)\in V_1\times V_2\simeq T^*_0\G_{qp}$
with $|\bar \t|(|P_{\bar\t}\bar\xi|+|\bar\eta|
)>0$. Consider the cylinder $\Lambda =\{(\xi, \eta, \t)\in V_1\times V_2: |\xi|^2+|\eta|^2=|\bar\xi|^2+|\bar\eta|^2\}\subset \G_{qp}\simeq T^*_0\G_{qp}$.
  Consider
the  map $(s,\xi, \eta,\t)\in\R\times\Lambda \mapsto \gamma(s,\xi, \eta, \t)\in V_1\times V_2$.
To prove the proposition, it suffices to show that $\det [\p_s \gamma, d_{\Lambda}\gamma]|_{(s, \xi, \eta, \t )=(\frac{2\pi}{|\bar\t|}, \bar\xi,\bar\eta,\bar\t)}=0$. Here $d_\Lambda \gamma\in\R^{N \times( N -1)}$ denotes   any family of independent  derivatives    in $T_{(\bar\xi, \bar\eta,\bar\t)}\Lambda$ with $N:=\operatorname{dim}\G_{qp}=qp+p+q$. If   $\bar \eta\neq 0$, choose $v\in\R^p\setminus\{0\}$ such that $\langle v,\bar\eta\rangle=0$ and consider the derivative $D_v:=v\cdot\nabla_\eta$  (note that existence of such a    direction~$v$ needs  $p\geq 2$).
Since~\eqref{ovest} is radial in $\eta$, it turns out easily that $D_v \gamma(\frac{2\pi}{|\bar\t|}, \bar\xi,\bar\eta, \bar\t)=0$. This implies that the determinant above vanishes.
If instead $\eta=0$,
we have
\begin{equation*}
 \gamma\Big(\frac{2\pi}{|\bar\t|}, \bar\xi, 0, \bar\t\Big) = \lim_{\e\to 0}  \gamma  \Big(\frac{2\pi}{|\bar\t|},\bar \xi, \e \eta_0 , \bar\t\Big)
\end{equation*}
where $\eta_0 \in\R^p$ is a fixed nonzero vector. Thus, the point with $\bar\eta=0$ is conjugate, being a limit of a family of conjugate points (recall that $\gamma$ is a smooth map).
\end{proof}

\begin{proposition}[Conjugate points  for $q\geq 2$ and  $p=1$, sufficient condition]\label{piuno} Let $q\geq 2$ and  consider the model $\G_{q1}$.
 Given $(\xi, \eta, \t)\in V_1\times V_2$, assume that $\t\neq 0$ and $|\eta|^2+|P_\t\xi|^2>0$ and consider  the extremal $\gamma(\cdot, \xi ,\eta, \t)$.
 If
 \begin{equation}\label{comprendo}
  \eta P_\t^\perp\xi= 0,
 \end{equation}
then the time $\bar s:= \frac{2\pi}{|\t|}$ is conjugate.
\end{proposition}
Note that assumption \eqref{comprendo} does not appear in Proposition~\ref{massetto}, where $p\geq 2$. We shall prove in Theorem~\ref{taglia}  that condition~\eqref{comprendo} is also necessary to have $\frac{2\pi}{|\t|}$ as a conjugate time. We do not discuss the case $p=q=1$: this is the familiar Heisenberg group~$\H^1$, and in this case, using rotation invariance around $t$-axis, it is easy to see that all cut points are conjugate.   Extremals in the Heisenberg group $\G_{11}$ can be read in plenty of papers,  e.g.~\cite{Monti00,AmbrosioRigot04,HajlaszZimmermann}. See also the general discussions in Section 13.2, 13.3 of \cite{AgrachevBarilariBoscain}.

\begin{proof}
Let $(\bar \xi, \bar \eta,\bar \t)\in V_1\times V_2\simeq T^*_0\G_{qp}$
with $|\bar \t|(|P_{\bar\t}\bar\xi|+|\bar\eta|
)>0$. As in the proof of Proposition~\ref{massetto} above, consider the cylinder $\Lambda =\{(\xi, \eta, \t)\in V_1\times V_2: |\xi|^2+\eta^2=|\bar\xi|^2+ \bar\eta^2\}\subset \G_{q1}\simeq T^*_0\G_{q1}$.
  Consider
the  map $(s,\xi, \eta,\t)\in\R\times\Lambda \mapsto \gamma(s,\xi, \eta, \t)\in V_1\times V_2$.
To prove the proposition, it suffices to show that $\det [\p_s \gamma, d_{\Lambda}\gamma]|_{(s, \xi, \eta, \t )=(2\pi/|\bar\t| , \bar\xi,\bar\eta,\bar\t)}=0$. Here $d_\Lambda \gamma\in\R^{N \times( N -1)}$ denotes   any family of independent  derivatives    in $T_{(\bar\xi, \bar\eta,\bar\t)}\Lambda$ with $N:=\operatorname{dim}\G_{q1}=2q+1 $.  Consider the vector field $Z:=-\langle\xi,\t\rangle\p_\eta+ \eta D_\t$, where $D_\t=\langle \t,\nabla_\xi\rangle$. Note that $Z$ is tangent to~$\Lambda$ and $Z\neq 0$, because $\eta^2+|P_\t\xi|^2\neq 0$.
Taking the form~\eqref{ovest} of $\gamma(\frac{2\pi}{|\t|},\xi,\eta,\t)$ into account, after a computation we get
   \begin{equation}\label{bottiglia}
Zx\Big(\frac{2\pi}{|\bar \t|},\bar \xi,\bar \eta,\bar \t\Big) =(-\langle\bar \xi,\bar \tau\rangle\p_\eta+\bar \eta D_{\bar \t})\frac{2\pi}{|\bar \t|}P_{\bar \t}^\perp\bar \xi=
\frac{2\pi}{|\bar \t|}P_{\bar  \t}^\perp\bar \t=0                                                   \end{equation}   and  $Zy(\frac{2\pi}{|\bar \t|},\bar \xi,\bar \eta,\bar \t)=0$. Also,
\begin{equation}\label{bottiglia2}
\begin{aligned}
 Zt\Big(\frac{2\pi}{|\bar \t|},\bar \xi,\bar \eta,\bar \t\Big)  & =\frac{\pi}{|\bar \t|^3}(-\langle\bar \xi,\bar \tau\rangle\p_\eta+\bar \eta D_{\bar \t})\Big((|P_{\bar \t}\bar \xi|^2+\bar \eta^2
)\bar \t-2P_{\bar \t}^\perp\bar \xi\bar \xi^T\bar \t\Big)
\\&
=\frac{\pi}{|\bar \t|^3}\Big\{-2\bar \eta\langle\bar \xi,\bar \t\rangle\bar \t +2\bar \eta  \langle P_{\bar \t}\bar \xi, D_{\bar \t} P_{\bar \t}\bar \xi\rangle \bar \t -2\bar \eta P_{\bar \t}^\perp \bar \xi D_{\bar \t}\langle \bar \xi,\bar \t\rangle\Big\}
\\& =-\frac{2\pi\bar \eta}{|\bar \t|}P_{\bar \t}^\perp\bar \xi,
\end{aligned}
\end{equation}
because  $D_{\bar \t} P_{\bar \t}^\perp\bar \xi = P_{\bar \t}^\perp \bar \t=0$ and terms along $\t$ cancel.
From these computation, we see that assumption $\bar \eta P_{\bar \t}^\perp\bar \xi=0$ implies that   $\det [\p_s \gamma, d_{\Lambda}\gamma]|_{(s, \xi, \eta, \t )=(\frac{2\pi}{|\bar\t|}, \bar\xi,\bar\eta,\bar\t)}=0$.
\end{proof}

  Propositions \ref{massetto} and \ref{piuno}  imply estimate $t_\cut\leq  2\pi/|\t| $ if $p\geq 2$. If $q\geq 2$ and  $p=1$, the upper estimate follows only if $\eta P_\t^\perp\xi= 0$.

\begin{remark} The upper estimate in case   $q\geq 2$,     $p=1$ and $\eta P_\t^\perp\xi\neq 0$  can be obtained easily
from the fact that if $\eta\neq 0$, then $\gamma(\frac{2\pi}{|\t|}, \xi, \eta, \t)=\gamma(
\frac{2\pi}{|\t|}, \xi, -\eta, \t)$. Thus there are two different extremals reaching the point
$\gamma(\frac{2\pi}{|\t|},\xi,\eta,\t)$
with length  $\frac{2\pi}{|\t|}|(\xi,\eta)|$.
\end{remark}
The remark above concludes the proof of the upper estimate in all cases.
In Section~\ref{dalbasso} we shall prove  the opposite estimate  $t_\cut\geq 2\pi/|\t|   $.

 In next proposition, we find for all $p,q$, all points $\gamma(\frac{2\pi}{|\t|}, \xi, \eta, \t)\in\G_{qp}$  reached by more than one  extremal of length $\frac{2\pi}{|\t|}|(\xi,\eta)|$.

 In the statement of the following proposition, to have a clean exposition, we take for granted equality $t_\cut=
2\pi/|\t| $ for all extremals different from Euclidean lines. (Note that arguments in Proposition~\ref{masso} and in  Section~\ref{dalbasso} are    independent of each other.)

  It is easy to show that for an extremal of the form $\gamma(s,\xi, \eta, \t)$ with $\t\neq 0$ and $\eta\neq 0$,  we have $ \gamma(\frac{2\pi}{|\t|}, \xi, \eta, \t)=\gamma(\frac{2\pi}{|\t|}, \xi, R\eta , \t)$ for all $R\in O(p)$. Therefore the point is reached by more than one  minimizer  of equal speed.   There are  also cut points with $\eta=0$ reached by at least two minimizers of equal speed, and we now characterize them. All other cut points are reached by a unique minimizer.

\begin{proposition}[Cut points reached by  at least  two different arclength minimizers] \label{masso}
Let $\t\in\R^q\setminus\{0\}$. Let  $|\xi|^2+|\eta|^2=1 $ and assume also that $|\eta|^2+ |P_\t\xi |^2>0$. Then there is $(\xi', \eta',\t')\neq (\xi,\eta,\t)$ such that
\begin{equation}\label{joao}
 \left\{\begin{aligned}
&|\xi'|^2+|\eta'|^2=|\xi|^2+|\eta|^2=1,\quad  |\t'|=|\t|\quad\text{and}
\\&
  \gamma(\tfrac{2\pi}{|\t'|},\xi', \eta',\t')=\gamma(\tfrac{2\pi}{|\t|},\xi, \eta,\t)
\end{aligned}\right.
\end{equation}
if and only if
either  $\eta\neq 0$, or
\begin{equation}
\label{horror}
 \eta=0
\quad\text{ and }\quad   \xi^T\t\;\text{ is not orthogonal to }\; \ker P_\t^\perp \xi.
\end{equation}
\end{proposition}

\begin{remark}\label{ser} \begin{enumerate}[nosep]
    \item Assumption $\t\neq 0$ and $|\eta|^2+ |P_\t\xi |^2>0$ ensure that the curve $\gamma(s,\xi,\eta, \t)$ is not an Euclidean line contained in the plane $t=0$. See Remark~\ref{notflat}.
    \item The second assumption in~\eqref{horror}
implicitely implies that the columns $P_\t^\perp \xi_1,\dots, P_\t^\perp\xi_p$  of $P_\t^\perp\xi$ are linearly dependent in $\R^q $ (i.e.~, $\operatorname{rank}P_\t^\perp\xi <p$) and that $\xi^T\t\neq 0\in\R^p$.   If $p=1$, this means $P_\t^\perp\xi = 0$ and $P_\t\xi\neq 0$.

\item In terms of coordinates $(x,y,t)$ on $V_1\times V_2$, points reached by a unique unit-speed length-minimizer will be characterized in Section~\ref{cinque}.

\item If $q=1$, the group $\G_{1p}=\H^p$ is the familiar  $p$-dimensional Heisenberg group,  and it is easy to see that if $\t\neq 0$   and $\eta=0\in\R^{p\times 1}$,    condition~\eqref{horror} is  satisfied by any $\xi\in\R^{1\times p}\setminus\{0\}$.  In case $q=1$, $P_\t^\perp\xi\in\R^1$ vanishes trivially.

\end{enumerate}\end{remark}

\begin{proof}[Proof of Proposition \ref{masso}]
 Let $(\xi,\eta,\tau)\in  T^*_{0}(V_1\times V_2)\simeq V_1\times V_2$ be such that
  $|\xi|^2+|\eta|^2=1$.
 We first characterize all $(\xi',\eta',\t')$ satisfying \eqref{joao}.
By \eqref{ovest}, condition
 $x(\frac{2\pi}{|\t'|}, \xi',\eta',\tau')=x(\frac{2\pi}{|\t|}, \xi,\eta, \t)$  and $|\t|=|\t'|$ give
\begin{equation}\label{x}
  P_\tau^\perp \xi  = P_{\t'}^\perp\xi',
\end{equation}
while assumption $t(\frac{2\pi}{|\t'|}, \xi',\eta',\tau')=t(\frac{2\pi}{|\t|}, \xi,\eta, \t)$ gives  \begin{equation}\label{t}
\begin{aligned}
 &  (|P_\t\xi|^2 +|\eta|^2 )\tau -2 P_\t^\perp\xi \,\xi^T\t
=  (|P_{\t'}\xi'|^2+|\eta'|^2\Big)\tau' -2
P_{\t'}^\perp\xi'  \xi'^T \t'.\end{aligned}
\end{equation}
Now we write $P_\t^\perp\xi =P_{\t'}^\perp \xi'=:v=[v_1,\dots, v_p]$, where the column space of $v$ satisfies $\Im v\subset  \Span\{\t,\t'\}^\perp$.
We have then
\begin{equation*}
 |\xi|^2+|\eta|^2
  =1=|\eta|^2+
 |P_\t \xi|^2+|v|^2\quad\text{and}\quad
|\xi'|^2+|\eta'|^2
=1=|\eta'|^2+
|P_{\t'}  \xi'|^2+
 |v |^2.
\end{equation*}
 Thus $|\eta'|^2+
|P_{\t'}^\perp\xi'|^2
= |\eta|^2+ |P_\t\xi|^2 $.
Projecting orthogonally~\eqref{t}
along $\Span\{\t,\t'\}$, we get then
$
 ( |\eta|^2+|P_\t\xi|^2 )(\t-\t')=0
$.
The parenthesis can not vanish by assumption.  This forces $\t'=\t$. Thus,~\eqref{t} becomes
 $ P_\t^\perp\xi\, \xi^T\t  = P_{\t }^\perp \xi'\,\xi'^T\t$.
Passing to a shorter notation  write  $\a:=\frac{\xi^T \t}{|\t|}$ and $\a' :=\frac{ \xi'^T \t}{|\t|}\in\R^p$
(keep in mind that $|P_\t\xi|^2=\frac{|\xi^T\t|^2}{|\t|^2}$). Thus, $(\xi', \eta', \t')$ satisfies all conditions~\eqref{joao}  if and only if $\t'=\t$, $P_{\t}^\perp\xi' =P_\t^\perp\xi =:v $ and there are $\a'  $ and  $\eta'\in\R^p$  such that
\begin{equation}\label{bianca}
\left\{\begin{aligned}
 & v\a'=v\a
 \\&|\eta'|^2+ | \a'| ^2 =|\eta|^2+|\a|^2= 1- |v |^2,
\end{aligned}\right.
\end{equation}
where $v=[v_1\dots, v_p]\in\R^{q\times p}$.
In order to conclude the proof, we need to understand for which given   $\eta,\alpha\in\R^p$ the pair  $\eta',\alpha'\in\R^p
$ satisfying~\eqref{bianca} must be chosen uniquely in the form $\eta'=\eta$
and $\alpha'=\a$. Note first that if $\eta\neq 0$, any choice $\eta'=R\eta$ with $R\in O(p)$, $R\neq I_p$ gives a solution $\eta'=R\eta$ and $\a'=\a$ different from $\eta,\a$  (if $p=1$, just choose $\eta' =-\eta$).  Therefore, to have uniqueness it must be $\eta=0$. If $\eta=0$, then~\eqref{bianca} becomes
\[
 v\a'=v\a\;\text{and}\; |\eta'|^2+ | \a'| ^2 = |\a|^2= 1-  |v |^2.
\]
If $\ker v=0$, then it must be $\a'=\a$, $\eta'=\eta$,  and we have uniqueness   (if $p=1$, this occurs when $P_\t^\perp\xi\neq 0\in\R^q$).   Let now $\ker v$ be nontrivial.
If $\a\in(\ker v)^\perp$, then $\a=\a_{\textup{LS}}$,  the  least-squares solution of the system $v\b=v\a$ with unknown $\b\in\R^p $. Then, $|\a'| \gneqq | \a|$ for all $\a'\neq \a$ solving $v\a'=v\a$. Then, it must be $\a'=\a =\a_{\textup{LS}}$ and we have again uniqueness in the choice of $\eta',\a'$.
Finally, if $\a\notin (\ker v)^\perp$, we can choose $\eta'=0$ and $\a'=2\a_{\textup{LS}}-\a\neq 0$. In this case $\a'\neq\a$, $|\a'|=|\a|$ and $v\a'=v\a$ as required.   If $p=1$ and $q\geq 2$, $\ker v\neq 0$ means  $v=0$. Then,  $\a_{\textup{LS}}=0$ and the non uniqueness choice becomes $\a'=-\a$.

To resume,
if $\eta\neq 0$, we can choose $\eta'=-\eta$, $\a'=\a$ and we have found a choice of $(\xi', \eta', \t')\neq(\xi,\eta,\t) $. If $\eta=0$, by assumption~\eqref{horror}, given $\eta=0$, $v=[v_1,\dots, v_p]\in\R^{q\times p}$ and $\a\neq 0$, there
is a second solution $\eta'=0$, $\a'=2\a_{\textup{LS}}-\a\neq\a$, where $\a_{\textup{LS}}$ solves $v\a_{\textup{LS}}=v\a$ and $\a_{\textup{LS}}\perp\ker v$.

We have ultimately proved that there is nonuniqueness if and only if either $\eta\neq 0$, or $\eta= 0$ and $\a\notin(\ker v)^\perp$. The proof is concluded.
 \end{proof}

 A careful inspection of the proof   above  shows that, if $p\geq 2$, when  the choice of $(\xi',\eta',\t')$ is not  unique, then  we have  an infinite,  continuous family of choices. If instead $p=1$ and $q\geq 2$, then in case of non uniqueness there are either two or infinitely many minimizers.

 We conclude this section with an example showing that the function $(\xi,\eta,\t)\in \G_{qp}\mapsto t_\cut(\xi,\eta,\t)$ is discontinuous for $q\geq 2$.
 \begin{example}
  \label{discontinua}
  Let   $q\geq 2$, let $\xi\in\R^{q\times p}\setminus\{0\}$   with $\Im\xi\subsetneq\R^q$   and $\t\in (\Im\xi)^\perp\setminus\{0\}$ (i.e.~$\xi=P_\t^\perp\xi$). Assume that $|\t|=1$ and consider $\eta\in\R^p\setminus\{0\}$. Then for any $\e>0$ we have the family of cut points
   \begin{equation*}
   \gamma(2\pi,\xi,\e\eta,\t)
   =
   (2\pi\xi,0,\e^2\pi|\eta|^2\t)\to (2\pi\xi,0,0)=\gamma(2\pi,\xi,0,\t).
  \end{equation*}
The corresponding cut-times are  $t_{\cut}(\xi, \e\eta,\t)=2\pi$ for all $\e>0$. However, $t_\cut(\xi, 0,\t)=+\infty$.
Note that the point $(2\pi\xi,0,0)$ in this example is abnormal.
 \end{example}

\section{Lower bound \texorpdfstring{$t_{\cut }\geq \frac{2\pi}{|\t|}$}{tcut geq 2pi/|tau|}}
 \label{dalbasso}
Here we prove the following lower bound for the cut-time.
\begin{proposition}\label{superbasso}
 Let $\bar u=u(\cdot, \bar\xi,\bar\eta,\bar\t)$ be a given extremal control with $\bar\t\neq 0$ and
 $|\bar\eta|+|P_{\bar\t}\bar\xi|>0$. Then the cut time of $\gamma_{\bar{u}}$ satisfies $t_\cut\geq \frac{2\pi}{|\bar\t|}$
\end{proposition}
\begin{proof}
 We work following the argument of~\cite[Section~2.3.2]{BarilariBoscainGauthier12}, which is in turn based on \cite[Chapter~12.4]{AgrachevSachkov04}.  Let $\bar\xi,\bar\eta,\bar\t$ be given and let  \begin{equation}\label{barrato}
(\bar  u(s),\bar v(s))= (u,v)(s,\bar\xi,\bar\eta,\bar\t) =e^{-s A_{\bar\t}}
(\bar \xi, \bar \eta)                                                                                                                                                                                    \end{equation}   and $\bar\gamma(s)=\gamma(s,\bar\xi,\bar\eta,\bar\t)$ be a given extremal path  which we assume to be parametrized by  arclength, i.e.~$|\bar\xi|^2+|\bar\eta|^2=1$. Fix a positive time $
 \hat s<\frac{2\pi}{|\bar\t|}$.  We want to show that $\bar\gamma$ minimizes length between $(0,0,0)$ and $( \bar x, \bar y,\bar t):=\bar \gamma(\hat s) $.

Consider a control $(u,v)\in L^2((0,\hat s), \R^{q\times p}\times\R^p) $ and  the corresponding path $\gamma_{(u,v)} $ as defined in~\eqref{senza}.
Denote $\gamma_{(u,v)}(  s)=( x_{(u,v)}(s),y_{(u,v)}(s), t_{(u,v)} (s)  ) $ for all $s$. On the control $(u,v)$ we require the following  three properties.
 \begin{enumerate}[noitemsep]
  \item[(1)] $|(u,v)  |=1$ a.e.~on $[0, \hat s]$ (i.e.~$\gamma_{(u,v)}$ is arclength).
  \item[(2)] We have $(x(\hat s ), y(\hat s ) )= (\bar x, \bar y) $.
  \item[(3)] $(u,v)$ maximizes the cost  $J(u,v):=\langle \bar\t, t_{(u,v)}(\hat s)  \rangle=:\int_0^{\hat s} \phi((x,y,t),(u,v)) $.

 \end{enumerate}
 We claim first that the three statements (i), (ii), and (iii) of Lemma 21 in \cite{BarilariBoscainGauthier12} hold in our setting too.

 Statement (i) claims that there is $(u,v)\in L^2$ such that (1), (2), and (3) hold. This is a standard compactness argument.   Roughly speaking, it suffices to take a minimizing sequence $\{(u_n, v_n)\}_{n\in\N }\in L^{\infty} ([0, \hat s ] , V_1  ) $. Using (1) and the ODE~\eqref{senza} one can easily check
 that the sequence $ \{ ( x_n, y_n, t_n )\}$ is equicontinuous and uniformly bounded. Passing to a subsequence, we may assume that $(x_n,y_n, t_n)$ converges uniformly to a Lipschitz function $(x,y,t)$  on $[0,\hat s]$. This ensures (2). To check that the limit $(x,y,t)$ satisfies~\eqref{senza}, we use the weak compactness of the sequence $\{  (u_n, v_n) \}$ which by (1) has a subsequence converging to a  limit $(u, v)$ satisfying $|(u(s), v(s))|\leq 1$ for a.~e.~$s\in[0,\hat s]$.
 Observe that the set $\{(u,v)\in V_1:|(u,v)|\leq 1\}$ is convex.

 Statement~(ii) claims that
$\gamma_{(u,v)}$ is a length-minimizer on $[0, \hat s]$. To show this property,  assume by contradiction that there is $\e>0$ and  an arclength control  $( u' , v' )  $ on $[0, \hat s -\e]$ such that $\gamma_{(u',v')}(0)=(0,0,0)$ and $
\gamma_{(u',v')}(\hat s -\e)=\gamma_{(u,v )}(\hat s)=(\bar x,\bar y, t_{(u,v)}(\hat s))$. Since the sub-Riemannian $\e$-ball  is open, we can extend $(u',v')$ on $[\hat s -\e, \hat s]$ to achieve a final point $(\bar x, \bar y, t_{(u',
v')}(\hat s))$ such that $\langle\bar\t, t_{(u',v')}(\hat s)\rangle >  \langle\bar\t, t_{(u,v)}(\hat s)\rangle$ contradicting (3).

Claim (iii) asserts that  the solution $\gamma_{(u,v)}$ discussed in (i) and (ii) has the precise form $\gamma_{(u,v)}= \gamma(\cdot, \xi, \eta, \lambda \bar\t)$ for a suitable $(\xi, \eta )$ of unit norm and $\la>0 $. To accomplish this step, observe that the cost function $\langle\bar\t, t_{(u,v)}\hat s\rangle$  is the integral on $[0,\hat s]$ of
\[
\begin{aligned}
\frac{d}{ds} \langle \bar\t, t\rangle & = \langle \bar \t, \dot t \rangle=\frac 12\langle\bar\t,
xv-uy\rangle
\\&=\frac 12\big(\langle x, \bar\t v^T\rangle-\langle y, u^T\bar \t\rangle\big)
=\frac 12 \langle(x,y), A_{\bar \t}(u,v)\rangle
=: \phi((x,y,t), (u,v)).
\end{aligned}
\]
Therefore, by \cite[Theorem~12.13]{AgrachevSachkov04},
 the Hamiltonian to study problem (1), (2) and   (3) is
\begin{equation*}\begin{aligned}
 H_{(u,v)}((x,y,t), (\xi, \eta,\t))&=
 \sum_{\substack{\a=1,\dots, q \\ k=1,\dots, p}}u_{\a k}\langle  X_{\a k}(x,y,t),(\xi,\eta,\t)\rangle +
 \sum_{j=1,\dots, p} v_j\langle  Y_j(x,y,t),(\xi,\eta,\t)\rangle
\\&
\qquad +2\nu\phi((x,y,t),(u,v))\\&=\langle u,\xi\rangle + \langle v,\eta\rangle -\frac 12\langle u,\t y^T\rangle +\frac 12\langle v, x^T\t\rangle + \nu  \big(\langle x, \bar\t v^T\rangle-\langle y, u^T\bar \t\rangle\big)                                                                                                                                                                                                                             \end{aligned}
\end{equation*}
where we used~\eqref{campetti}.
Since we are maximizing $\int_0^{\wh s}\phi$, we have $\nu\geq 0$. An optimal control for our problem should satisfy the related Hamilton equations for suitable $\nu$. Furthermore it must satisfy the \emph{transversality condition}, with target manifold $N_1:=\{(x,y,t)\in V_1\times V_2: (x, y)= (\bar x, \bar y)\}$. Thus, $(\xi, \eta, \t)(\wh s)\perp N_1$, which becomes $\tau(\wh s)=0$. Since $\dot \t = -\nabla_t H=0$, we have $\tau(s)=0$ on $[0, \hat s]$ for the requested solution.
Therefore (along an optimal control) we can write the Hamiltonian in the form
\begin{equation*}\begin{aligned}
 H_{(u,v)}((x,y,t), (\xi, \eta,\t))& =\langle u,\xi\rangle + \langle v,\eta\rangle   + \nu  \big(\langle x, \bar\t v^T\rangle-\langle y, u^T\bar \t\rangle\big)
 \\&
 =\big\langle(u,v), (\xi,\eta)-\nu A_{\bar\t}(x,y)\big\rangle.
\end{aligned}
\end{equation*}
The remaining Hamilton equations are
\begin{equation}\label{hha}
 \dot \xi =-\nabla_x H=-\nu \bar \t v^T\quad\text{ and }\dot\eta =-\nabla_yH=\nu u^T\bar\t, \quad \text{i.e. } (\dot\xi , \dot\eta)= -\nu A_{\bar\t}(u,v).
\end{equation}
The maximality condition \cite[Theorem 12.13, eq.~(12.31) and Remark 12.2]{AgrachevSachkov04} states that along the  optimal control $(u(s), v(s))$ we have for  almost all  $s\in[0,\hat s]$
\begin{equation*}
\begin{aligned}
 \Big\langle(u(s),v(s)), & (\xi(s),\eta(s) )-\nu A_{\bar\t}(x(s),y(s))\Big\rangle
 \\&=\max_{|(u,v)|=1}
  \Big\langle(u ,v ), (\xi(s),\eta(s) )-\nu A_{\bar\t}(x(s),y(s))\Big\rangle
  \\& =\Big\|(\xi(s),\eta(s) )-\nu A_{\bar\t}(x(s),y(s)) \Big\|=1.
\end{aligned}
\end{equation*}
Therefore, the unit-norm control $(u(s), v(s))$ has the form
$
 (u(s), v(s))=(\xi(s),\eta(s) )-\nu A_{\bar\t}(x(s),y(s)).
$ This shows first that $(u(0), v(0))=(\xi(0),\eta(0))$.  Furthermore, since $(\dot x, \dot  y)=(u,v)$, differentiating the latter formula we get
$ (\dot u ,\dot v )=(\dot\xi ,\dot\eta )-\nu A_{\bar\t}(u , v)=-2\nu A_{\bar\t}(u,v)=- A_{2\nu\bar\t}(u,v),$
where we also used~\eqref{hha}. Thus the solution has the form
\begin{equation}\label{controllo}
(u(s), v(s))= e^{-sA_{2\nu\bar\t}}(\xi,\eta)                                                                                                            \end{equation}
  where $\nu>0$ and $(\xi, \eta)$ is a unit vector. This shows (iii)  and completes the analogous of   \cite[Lemma 21]{BarilariBoscainGauthier12}.

We are left with the proof that the control $(\bar u, \bar v)$ in~\eqref{barrato} and $(u,v)$ in~\eqref{controllo} are the same, i.e.~ that $\nu=\frac 12$. Precisely, we prove the following claim (see again~\cite{BarilariBoscainGauthier12}).

\textbf{Claim.} Let $(\bar u, \bar v) :\R\to\R^{q\times p}\times\R^p$  be
the control in~\eqref{barrato}. Let   $(\xi, \eta)\in V_1$ be of unit norm,  let $\hat s\in \big]0, \frac{2\pi}{|\bar\t|} \big[$ and let $\nu>0$. Let $(u,v)(s)=e^{-sA_{2\nu \bar\t}}(\xi,\eta)$ and put $(x,y)(s)=\int_0^s(u,v)$. Then,
if
 $(x(\hat s), y(\hat s))= (\bar x, \bar y)$ and $\gamma_{(u,v)}$ minimizes length at least until~$\hat s$,
we have $\nu=\frac 12$ and $(\xi,\eta)=(\bar\xi,\bar\eta)$.

To prove the claim   note first that, since $\gamma_{(u,v)}$ minimizes length at least until~$\hat s$ it must be $\hat s\leq   \frac{2\pi}{2\nu|\bar\t|}$, which gives the   upper    bound on $\nu\leq  \frac{\pi}{\hat s|\bar\t|}$.
 Write now   $\xi =P_{\bar\t}\xi +P_{\bar\t}^\perp\xi =\frac{\bar\t}{|\bar\t|}\la^T  +v $ where $v =P_{\bar\t}^\perp\xi $ and $\lambda=\frac{\xi^T\bar\t}{|\bar\t|}$. In view of~\eqref{pari}, assumption $ (x(\hat s), y(\hat s))= (\bar x, \bar y)$ reads
\begin{equation*}
\begin{aligned}
& \hat s T ( \hat s\nu|\bar\t| ) \frac{\bar\t}{|\bar\t|}\Big\{\la^T \cos( \hat s\nu|\bar\t| )
-\eta^T \sin( \hat s\nu|\bar\t| )\Big\}+\hat sP_{ \bar\t }^\perp\xi =\bar x
\\& \hat s T( \hat s\nu|\bar\t| )\Big\{\eta\cos( \hat s\nu|\bar\t| )+\lambda \sin(\hat s\nu|\bar\t|)\Big\}=\bar y\end{aligned}
\end{equation*}
Recall that $\bar\t$ is given. Then we can project the system along $\bar\t^\perp$ and $\bar\t$, obtaining
$P_{\bar\t}^\perp\xi =\frac{1}{\hat s}P_{\bar\t}^\perp \bar x $ and
\begin{equation}\label{cramer}
\left\{\begin{aligned}
& \hat s T ( \hat s\nu|\bar\t| )\Big\{\la \cos( \hat s\nu|\bar\t| )-\eta \sin( \hat s\nu|\bar\t| )\Big\} =\frac{\bar x^T\bar\t}{|\bar\t|}
\\& \hat s T( \hat s\nu|\bar\t| )\Big\{\eta\cos( \hat s\nu|\bar\t| )+\lambda\sin(\hat s
\nu|\bar\t|)\Big\}=\bar y.\end{aligned}\right.
\end{equation}
Taking the norm and summing up we find
$ \hat s^2 T ( \hat s\nu|\bar\t| )^2\big(|\la|^2+|\eta|^2\big)= \frac{|\bar x^T\bar\t|^2}{|\bar\t|^2}
 +|\bar y|^2
$.
Since $|\la|^2+|\eta|^2= |P_{\bar\t}\xi |^2+|\eta|^2=
1-  |P_{\bar\t}^\perp \xi |^2 =
1- \frac{|P_{\bar\t}^\perp \bar x |^2}{\hat s^2}$, we find
\begin{equation}\label{unic}
\hat s^2\Big( \frac{\sin(\nu\hat s|\bar\t|)}{ \nu\hat s|\bar\t| }\Big)^2
\Big(1- \frac{|P_{\bar\t}^\perp \bar x |^2}{\hat s^2}\Big)= \frac{|\bar x^T\bar\t|^2}{|\bar\t|^2}
+|\bar y|^2.
\end{equation}
Since $\hat s , \bar x, \bar y $ and $\bar \t$ are known, the only unknown here   is $\nu \in \left] 0, \pi/
 (\hat s |\bar\t|  )\right]$.    We already know that $\nu=\frac 12$  belongs to that interval and   is a solution of the equation~\eqref{unic}. Since the function $\nu\mapsto ( \sin(\nu\hat s|\bar\t|) /(\nu\hat s|\bar\t| ) )^2$ is strictly
decreasing on $\left] 0, \pi/
 (\hat s |\bar\t|  )\right]$, the solution $\nu=\frac 12$ is unique. Letting then $\nu=\frac 12$ we go to the Cramer system~\eqref{cramer} and we find uniquely $\eta=\bar\eta$ and $\la=\frac{\bar\xi^T\bar\t}{|\bar\t|}$.
The proof is finished.
\end{proof}
 \begin{remark}We give an idea of our choice of the cost~(3) in the proof of Proposition~\ref{superbasso}.
 Our choice  is suggested by the   form of the cost in problem~(P) at page 570  of~\cite{BarilariBoscainGauthier12}, which is actually a control problem in a corank-1 quotient of the general corank-2 group the authors are working with. In our case, to get corank-1 quotients, we can  take  any unit vector $\omega\in\R^{q\times 1}$ and we can consider  the quotient $\G_\omega$ defined as follows.
$\G_\omega =\{(x,y, \omega\omega^T t)\in \G_{qp} : (x,y,t )\in \G_{qp}\}$. As a set, $\G_\omega= V_1\times \Span\{\omega\}$.  It can be equipped  with the operation    $(x,y, \lambda\omega)\cdot(\xi, \eta, \mu\omega)=
\Big(x+\xi, y+\eta, \omega(\lambda+\mu+\frac{1}{2}\omega^T(x \eta-\xi y) )\Big)$. This turns out to be a sub-Riemannian Carnot group of corank-1.  Now, taking the extremal control~$(\bar u, \bar v)$ in~\eqref{barrato} and the corresponding extremal $\bar \gamma(\cdot,\bar\xi,\bar\eta,\bar\t) = (x,y, t) (\cdot, \bar\xi,\bar\eta,\bar\t)$,  if we project  such extremal on a quotient $\G_\omega$, it can be checked that this gives an extremal control in $\G_\omega$ if and only if $\omega =\frac{\bar\t}{|\bar\t|}\in\Span\{\bar\t\}$. In such case, the vertical coordinate in~$\G_\omega=\G_{\bar\t/|\bar\t|}$ is proportional to  the cost appearing in~(3).
\end{remark}
\color{black}

 \section{Description of the cut locus  }\label{cinque}
  In this section we identify precisely the cut-locus and in a significant example we discuss some of its regularity properties.   Some of the results of this section will be used in the proof of Theorem~\ref{taglia} in Section~\ref{cinese}.


  Item \ref{duee} of  Theorem~\ref{cutolo} has been proved in Proposition~\ref{massetto}. Next we prove the remaining ones.
\begin{proof}[Proof of Theorem~\ref{cutolo}, items  \ref{unno},  and \ref{terzi} and~\ref{tres}]
 Let $\gamma(\cdot, \xi, \eta, \t)$ be the extremal appearing in Proposition~\ref{monte}. We know that
 \begin{equation}\label{cinquantuno}
\begin{aligned}
  \Cut(\G_{qp}) &= \Big\{\gamma\Big(\frac{2\pi}{|\t|}, \xi, \eta, \t\Big):(\xi, \eta, \t)\in
  V_1\times V_2,\; \tau\neq 0\text{  and }|\eta|^2+ |P_\t\xi |^2>0\Big\}
  \\ & =\{\gamma(2\pi,\xi,\eta,\t):(\xi, \eta, \t)\in  V_1\times V_2,
  \;  \;   |\tau|=1
  \text{  and }|\eta|^2+ |P_\t\xi |^2>0\},
\end{aligned}
 \end{equation}
 by~\eqref{benny}.

 Step 1.
 We show first that the set~\eqref{cinquantuno} is contained in the set~\eqref{etichetta}. By~\eqref{ovest}, a point $(x,0,t) $
in~\eqref{cinquantuno} has the form
 \begin{equation*}
\left\{\begin{aligned}
  x &=2\pi \,P_\t^\perp\xi
  \\
  t&=\pi\Big(|\eta|^2+ |P_\t\xi |^2\Big)\t -2\pi    P_\t^\perp\xi \,\xi^T\t.
\end{aligned} \right.
 \end{equation*}
The first line tells immediately that $\t\perp\Im x$. Furthermore, it must be
 $|\eta|^2+ |P_\t\xi |^2>   0$, see~\eqref{cinquantuno}.
Eliminating $P_\t^\perp\xi $ we get
$
 t=\pi\Big(|\eta|^2+ |P_\t\xi |^2\Big)\t - x\,\xi^T\t$. Thus
\begin{equation*}
\pi\Big( |\eta|^2+ |P_\t\xi |^2 \Big)\t=P_{\Im x}^\perp t\quad\text{and}\quad -x\,\xi^T\t =P_{\Im x} t.
\end{equation*}
Letting $\b=\xi^T \t \in\R^{p\times 1} $,  we get
$
\pi\Big( |\eta|^2+|\b|^2 \Big)\t=P_{\Im x}^\perp t$ and $-x\b =P_{\Im x} t
$.
From last formula we get immediately
\begin{equation}\label{quadro}
\begin{aligned}
 |P_{\Im x}^\perp t|  = \pi (|\eta|^2+|\b|^2) \geq \pi|\b|^2 &\geq   \pi\min\{|\b'|^2:-x\b'=P_{\Im x}  t\}
 \\&=: \pi |x^\dag P_{\Im x}  t|^2= \pi|x^\dag t|^2.
\end{aligned}
\end{equation}
The inclusion is proved.

 Step 2. Now we show the opposite inclusion and we characterize which  cut points are reached by a unique minimizer.

Let $(x,0,t) $ be in the set in the right-hand side of~\eqref{etichetta}. We must find $(\xi, \eta)\in V_1 $ and $\t\in V_2 $   such that
\begin{equation}\label{ventisei}
\left\{\begin{aligned}
 & 2\pi\,P_\t^\perp   \xi= x
 \\&
\pi\Big(|\eta|^2+ |P_\t \xi |^2\Big)\t-2\pi  P_\t^\perp \xi \,\xi^T\t =t
\\&
  |\t|=1
  \quad\text{ and }\quad
  |\eta|^2+ |P_\t\xi |^2>0.
\end{aligned}\right.
\end{equation}
Since $\t$ is a unit vector and from the  first line, we can write
$\xi = P_\t\xi +P_\t^\perp\xi =\t\t^T\xi+ \frac{x}{2\pi}=:-\t\la^T+\frac{x}{2\pi}$, where we put~$\lambda=-\xi^T\t\in\R^{p\times 1}$.
Thus, to find $\xi$, it suffices to know the vector
$\la=-\xi^T\t\in \R^p$. Concerning the vector  $\t$ we are looking for, it must be orthogonal to $\Im x$.
Then, projecting the second line of~\eqref{ventisei} along $(\Im x)^\perp$, we get
\begin{equation}\label{massimino}
 \pi\Big(|\eta|^2+ |P_\t \xi |^2\Big)\t=P_{\Im x}^\perp t\neq 0.
\end{equation}
Note that $P_{\Im x}^\perp t\neq 0$ by assumption. Therefore, since $\t$ is unit-norm, it must be first $  |\eta|^2+ |P_\t\xi |^2>0$ and furthermore  $\t =\frac{P_{\Im x}^\perp t}{|P_{\Im x}^\perp t|}
$.
Next project the second line of~\eqref{ventisei} along $\Im x$. We get the equation $x\la=P_{\Im x}t $. If   the columns $x_1, \dots, x_p$ of $x$ are independent, then $x^Tx $ is nonsingular and we find  a unique
$\lambda =(x^T  x)^{-1}x^T t=x^\dag t$ solving the problem. If instead
$x_1,\dots, x_p$ are dependent, the solutions $\la\in\R^p$ of $x\la = P_{\Im x}t$ form an affine space of the form $\{x^\dag t+\mu: \mu\in\ker x\}$, where $x^\dag t$ satisfies $xx^\dag t=P_{\Im x}t$ and  has the further property $x^\dag t\perp\ker x$, i.e. $x^\dag t$ it is the minimal-norm solution:
 $|x^\dag t|\leq |x^\dag t+\mu|$ for all $ \mu\in\ker x$.
More precisely, since $x^\dag t\perp \mu$, we have  $|x^\dag t+\mu|^2=|x^\dag t|^2+|\mu|^2$.
Fix now any $\mu\in \ker x$ and choose then $\lambda = x^\dag t+\mu$. Multiply  the second line of~\eqref{ventisei} scalarly by $\t$ (taking into account that $-\xi^T\t=\la=x^\dag t +\mu$). This gives
\begin{equation*}
 \pi(|\eta|^2+|x^\dag t+\mu|^2)=\langle t,\t\rangle =\Big\langle t,\frac{ P_{\Im x}^\perp t}{|
 P_{\Im x}^\perp
 t|}
 \Big\rangle=|P_{\Im x}^\perp t|.
\end{equation*}
Again by  orthogonality
$\mu \perp x^\dag t $
we get
\begin{equation}\label{ultima}
 \pi(|\eta|^2+|\mu|^2)=|P_{\Im x}^\perp t|-\pi|x^\dag t|^2.
\end{equation}
The right-hand side is nonnegative by assumption. Therefore, any choice of $\eta\in\R^p$ and $\mu\in\ker x$ such that~\eqref{ultima} is fulfilled, will provide a solution fulfilling~\eqref{ventisei}. This finishes the proof of the inclusion.

To prove item~\ref{terzi} of Theorem~\ref{cutolo},  observe that, given $(x,0,t)\in\Cut(\G_{qp})$, uniqueness of the choice of $(\xi, \eta, \t)$ holds if and only if the choice of $\eta$ and $\mu\in\R^p$
satisfying~\eqref{ultima} is unique. This happens if and only if the right-hand side of~\eqref{ultima} vanishes, i.e. $|P_{\Im x}^\perp t|=\pi|x^\dag t|^2$.

Step 3.
We finally check  formula~\eqref{pece}.
Let  $(x,0,t)=\gamma(2\pi, \xi, \eta, \t)\in \Cut(\G_{qp})$, where $|\t|=1$.
The   horizontal speed of the minimizer
$\gamma(\cdot , \xi, \eta, \t):[0, 2\pi]\mapsto V_1\times V_2$ is
$\sqrt{|\eta|^2+|\xi|^2}$ .
We have
$ |\eta|^2+|\xi|^2=|\eta|^2+ |P_\t \xi |^2+ |P_\t^\perp\xi |^2
$.
From the first line of~\eqref{ventisei} we have $ |P_\t^\perp\xi |^2= \frac{|x |^2}{4\pi^2} $.
From~\eqref{massimino} we find
$
  |\eta|^2+ |P_\t \xi |^2  =\frac{|P_{\Im x}^\perp t|}{\pi}
$.
Collecting formulas, we conclude that
\begin{equation*}
 d(x,0,t)^2=4\pi^2(|\xi|^2+|\eta|^2)=|x|^2+4\pi|P_{\Im x}^\perp t|,
\end{equation*}
as required. Note that  the component $ P_{\Im x} t$  does not appear in the distance.
\end{proof}

\begin{remark}\label{insiemedenso}
 Observe that the set of   points $(x,0,t)\in\Cut(\G_{qp})$ reached by more than one unit-speed length-minimizer is dense in $\Cut(\G_{qp})$. To see that, it suffices to take a point $(x,0,t)=(x,0, P_{\Im x}t+
 P_{\Im x}^\perp t) $ such that  $P_{\Im x}^\perp t\neq 0$ and satisfying  $|P_{\Im x}^\perp t|=\pi|x^\dag t|^2$. This point can be reached by a unique unit-speed minimizer. Next consider the family of approximating points $
  (x_\e, 0, t_\e) := (x, 0, P_{\Im x }t+(1+\e)P_{\Im x}^\perp t)$,
where $\e>0$.
We have easily $|P_{\Im x_\e}^\perp t_\e|=(1+\e)|P_{\Im x }^\perp t|$.
Furthermore, note that $x_\e^\dag t_\e$ only depends on $P_{\Im x_\e} t_\e =P_{\Im x}t$ and not on $P_{\Im x_\e}^\perp t_\e $ which changes with $\e$.
Thus
 $ x_\e^\dag t_\e =x^\dag t$ for all $\e>0$.
Thus the approximating point satisfies the strict inequality $|P_{\Im x_\e}t_\e|>\pi|x_\e^\dag t_\e|^2 $. Consequently, it can be reached by more than one length-minimizer.
\end{remark}

\begin{remark}\label{spiego}
As a byproduct of Step 2 of the proof above, and for future reference, observe that, if $(x,0,t)\in\Cut(\G_{qp})$
and if $\ker x$ is trivial, then we can write $(x,0,t)=\gamma(2\pi ,\xi,\eta,\t)$, where $\t=\frac{P_{\Im x}^\perp t}{|P_{\Im x}^\perp t|}$, $\xi=\t(x^\dag t)^T+\frac{x}{2\pi}$ and $\eta$ satisfies
\begin{equation}
 \label{spiegazione}
 \pi|\eta|^2=|P_{\Im x}^\perp t|-\pi|x^\dag t|^2.
\end{equation}

\end{remark}

\color{black}

\subsection{The group \texorpdfstring{$\G_{q1}$}{Gq1}} \label{zoppo}
Here $(x,0,t)\in\R^q\times\R\times\R^q$. All points of the form $(0,0,t)\in\G_{q1}$ with $t\neq 0$ belong to $\Cut(\G_{q1})  $. Let $(x,0,t)\in \G_{q1}$ be such that $x\neq 0$ and $t\neq 0$. The equation $x\beta=P_{\Im x} t$ has a unique real solution   $x^\dag (P_{\Im x} t)= \frac{\langle x,t\rangle}{|x|^2}$.
  Then
$|x^\dag t|^2 =|x^\dag P_{\Im x}t|^2=\frac{\langle x,t\rangle^2}{|x|^4}=\frac{|  P_{\Im x}t|^2}{|x|^2}$.
Thus we have
\[
\begin{aligned}
  \Cut(\G_{q1})&= \Big\{(x,0,t)\in\G_{q1}:t\neq 0\;\text{and}\; |x|^2\;|P_{\Im x}^\perp t|\geq \pi  |P_{\Im x}t|^2 \Big\}
 \\&=\{(x,0,t): t\neq 0 \;\,{and}\;\,  |x|^8|t|^2-|x|^6\langle t,x\rangle^2\geq \pi^2\langle t,x\rangle^4\},                                                                                                                        \end{aligned}
\] where, passing from the first to the second line, we used the Pythagorean theorem $|P_{\Im x}^\perp t|^2=
|t|^2- | P_{\Im x}t|^2= |t|^2-\frac{\langle t,x\rangle^2}{|x|^2}$.

In the next proposition, we analyze the regularity of the set of cut points where equality holds. Namely, of  \begin{equation}\label{sigm}
 \Sigma:=\{(x,0,t): t\neq 0 \;\,{and}\;\,  |x|^8|t|^2-|x|^6\langle t,x\rangle^2=\pi^2\langle t,x\rangle^4\}.
\end{equation}
\begin{proposition}\label{ciliegia}  Let $\Sigma \subset \Cut(\G_{q1})$ be the set in~\eqref{sigm}. Let   $\Sigma_0:=\{(x,0,t)\in\Sigma :x\neq 0\}\subset\Sigma$. Then $\Sigma_0$ is a smooth codimension-one  embedded  submanifold of $\{(x,0,t): (x, t)\in \R^q\times\R^q\}$.
Furthermore,  the whole $\Sigma$ is not a manifold.
 \end{proposition}
Observe that $\Sigma_0$ has codimension $2$ in $\G_{q1}$. The surface $\Sigma_0$ contains all points $(x,0,t)\in\Cut(\G_{qp})$ which are reached by a unique length-  minimizing unit-speed curve.

\begin{proof}To check this statement, it suffices to observe that $\Sigma_0$ is the zero-level set of the function $\psi:(\R^q\setminus\{0\})\times (\R^q\setminus\{0\})\to \R$
 \begin{equation}\label{piessei}
 \psi(x,t)=|t|^2-\frac{\langle x,t\rangle^2}{|x|^2} -\pi^2
 \frac{\langle x,t\rangle^4}{|x|^8}.                                    \end{equation} A short computation gives
\begin{equation*}
 \left\{
 \begin{aligned}
 &\nabla_t \psi = 2t -\frac{2}{|x|^2}\langle x,t\rangle
\Big\{1+2\pi^2\frac{\langle t, x\rangle^2}{|x|^6}\Big\}x
\\&
\nabla_x\psi =
-\frac{2}{|x|^2}\langle x,t\rangle
\Big\{1+2\pi^2\frac{\langle t, x\rangle^2}{|x|^6}\Big\}t+
\frac{2}{|x|^4}\langle x,t\rangle^2
\Big\{1+\frac{4\pi^2}{|x|^6}\langle x,t\rangle^2\Big\}x.\end{aligned}
\right.
\end{equation*}
Observe now that  at all point of $\Sigma_0$, we have $\langle x,t\rangle\neq 0$ (see~\eqref{piessei}). Furthermore, $x$ and $t$ are independent (otherwise, again~\eqref{piessei} fails).  Denoting $\nabla_t \psi=:a t+bx$ and $\nabla_x \psi=: bt+cx$,  we see that the Jacobian of $\psi$  has full rank  if and only if $ac-b^2\neq 0$.
A computation gives
\begin{equation*}
 ac-b^2=-\frac{16\pi^4}{|x|^{16}}\langle x,t\rangle^6\neq 0
\end{equation*}
for any $(x,t)\in\Sigma_0$. Thus, $\Sigma_0$ is a smooth manifold.

To show that  the whole $\Sigma$ is not a manifold, we argue by contradiction. Without loss of generality, consider $q=2$ and take the point $(x,t)=(0, e_2)=(0,0,0,1)$. Assume that $\Sigma$ is a smooth  embedded  manifold in a neighborhood of $(0, e_2)$.

\textit{Step 1.} If $\Sigma$ is a smooth surface containing, $(0, e_2)$, a tangent vector $(\xi, \t)\in T_{(0, e_2)}\Sigma $, would have the form $(\xi, \t)=(x'(0), t'(0))$, where $s\mapsto (x(s), t(s))$
is a curve belonging to $\Sigma$ for $s\in(-1,1)$, $(x(0), t(0))= (0, e_2)$ and $(x'(0), t'(0))=(\xi,\t)$. We claim that all such vectors satisfy $\xi\perp e_2$. This implies that  $T_{(0, e_2)}\Sigma = \{(\xi_1, 0, \t_1,\t_2): \xi_1, \t_1,\t_2\in\R\}$. To prove the claim, take a curve $(x(s), t(s))$ as above
and expand at the first order $x(s)= x' s+o(s)$ and $t(s)= e_2+t's+o(s)=e_2+o(1)$, where $(x', t')=(x'(0), t'(0))$, $o(1)\to 0$ and $\frac{o(s)}{s}\to 0$ as $s\to 0$. Then $|x(s)|^2= |x'|^2 s^2+o(s^2)$, $|t(s)|^2 = 1+o(1)$, while $\langle x(s), t(s)\rangle =\langle x', e_2\rangle s +o(s)= x'_2 s+o(s)$.
Inserting into~\eqref{sigm}, we get
\begin{equation*}
 (|x'|^8s^8 + o(s^8))(1+o(1))-(|x'|^6 s^6 +o(s^6))((x'_2)^2 s^2+o(s^2))-
 \pi^2((x_2')^4 s^4+o(s^4))=0.
\end{equation*}
Comparing powers of $s$, we get  $x_2' =\langle x'(0), e_2\rangle =0$, as desired.

\textit{Step 2.} Being $T_{(0, e_2)}\Sigma =\{(\xi_1, 0, \t_1,\t_2):\xi_1,\t_1,\t_2\in\R\}$,  the manifold $\Sigma $ can be written as a   graph of a function $F $ of the variables $x_1, t_1, t_2$. Namely, there are neighborhoods  $V:=\{(x_1, t_1, t_2)\in\R^3:
| (x_1, t_1, t_2)-(0,0,1)|<\e\} $ and $W=\{x_2\in\R:|x_2|<\delta\}$ such that  we have
\[
(V\times W)\cap \Sigma  =  \{(x_1, F(x_1, t_1, t_2), t_1, t_2):  (x_1, t_1, t_2)\in V\}.
\]
 Precisely, for all $(x_1, t_1, t_2)\in V$ there is a unique $F=F(x_1, t_1, t_2)\in\left]-\d,\d\right[$ such that $\psi(x_1, F, t_1, t_2)=0$. Let us test such property on points of the form $(x_1, 0, t_2)\in V$.    Letting $F=F(x_1, 0, t_2)$, we have after some simpifications
 \begin{equation}\label{q}
\psi(x_1, F, 0, t_2)=t_2^2 -\frac{(t_2 F)^2}{x_1^2+F^2}-\pi^2\frac{(t_2 F)^4}{(x_1^2+F^2)^4}=0.
 \end{equation}
Since $F$ appears quadratically, its uniqueness  gives $F(x_1, 0, t_2)=0$ as soon as $x_1  t_2\neq 0$. However, inserting $F=0$ into~\eqref{q}, we get $  t_2^2=0$ for all $(x_1, 0, t_2)\in V$, which gives a contradiction.
\end{proof}

\section{Lower estimate of conjugate time for \texorpdfstring{$p=1$}{p=1}}
 \label{cinese}
In this section we show that for $q\geq 2$ and $p=1$, in the model $\G_{q1}$ there are minimizers such that the $t_\cut$ is strictly smaller than the first conjugate time. The following statement completes Proposition~\ref{piuno}. Such phenomenon has already been encountered in~\cite{BarilariBoscainGauthier12}.
\begin{theorem}[Conjugate points, $q\geq 2$, $p=1$, necessary condition]
\label{taglia}
Let $q\geq 2$ and  consider the model $\G_{q1}$.
 Given $(\xi, \eta, \t)\in V_1\times V_2$, assume that $\t\neq 0$ and $|\eta|^2+|P_\t\xi|^2>0$ and consider  the extremal $\gamma(\cdot, \xi ,\eta, \t)$.
 Then, if $\bar s:= \frac{2\pi}{|\t|}$ is a conjugate time of $\gamma(\cdot, \xi,\eta,\t)$, it must be
 \begin{equation}\label{noncomprendo}
  \eta P_\t^\perp\xi =0.
 \end{equation}
\end{theorem}
We do not discuss the case $p=q=1$, the lower dimensional Heisenberg group, where it is known that cut time and first conjugate times are always the same.
\begin{proof} Consider
  $\G_{q1}$  and a point $(\xi,\eta,\t)\in T^*_{(0,0,0)}\G_{q1}$. Define $\Lambda =\{(\xi' , \eta',\tau')\in T^*_{(0,0,0)}\G_{q1}:|\xi'|^2+{\eta'}^2=|\xi|^2+\eta^2 \}\subset T^*_{(0,0,0)}\G_{q1}$.
  It suffices to show that $\eta|P_\t^\perp\xi|\neq 0$ implies that $s=\frac{2\pi}{|\t|}$ is not conjugate.
  Fix an orthogonal frame in $T_0^* \G_{q1}\simeq\R^q$ choosing
\begin{equation}\label{duale}
\t ,\quad   \omega_{q-1}:=P_\t^\perp\xi,\quad\text{and}\quad \omega_1, \dots, \omega_{q-2}\in\Span\{\t, P_\t^\perp\xi\}^\perp.  \end{equation}
  Note that if $q=2$, the frame will be formed by $\omega_1=P_\t^\perp\xi$ and $\t$ only and the discussion will be easier (see below).
Fix also the  frame of vector fields \[ Z=-\langle\xi,\tau\rangle\p_\eta+\eta D_{\t} \qquad V_k=
-\langle\xi,\omega_k\rangle\p_\eta+\eta D_{\omega_k}\quad\text{where}\quad  k\leq q-1
\] and $D_{\omega_k}:=\langle\omega_k, \nabla_\xi\rangle$ and $D_\t=\langle\t,\nabla_\xi\rangle$. Observe that the frame is tangent to~$\Lambda$. \footnote{     This comes from the fact the set $\Lambda$ is  defined by the equation
$F(\xi, \eta,\t):=|\xi|^2+ \eta^2 =$ constant, and we have $ZF=V_kF=0$ identically.  }
Furthermore, the frame in independent,  because $\eta\neq 0$ and $\omega_1,\dots, \omega_{q-1}, \t$ are independent. We need to show that the following determinant is nonzero.
\begin{equation*}
\det M : =   dx_1\wedge\cdots\wedge dx_q\wedge dy\wedge dt_1\wedge\cdots\wedge dt_q(\p_s\gamma, Z\gamma, V_1\gamma, \dots, V_{q-1}\gamma,d_\t\gamma)\neq 0
\end{equation*}
Let $\ell:= \xi\cdot dx +\eta dy+\t\cdot dt$ be the Liouville form. Since $|\t|\neq 0$, up to a nonzero factor we change  $dx\wedge dy\wedge dt_1\wedge\cdot\wedge d t_q$ with $dx\wedge dy\wedge \ell\wedge \mu_1\wedge\cdots\wedge \mu_{q-1}$,   where $\mu_k:=\omega_k\cdot dt $.
By Lemma~\ref{liouville}, we have $\ell(Z\gamma)=\ell(V_k\gamma)=0$,  for all $k=1,\dots, q-1$, while $\ell(\p_s\gamma)=1$.
\begin{equation*}
\begin{aligned}
\det M & \sim  \ell(\p_s\gamma)(dx\wedge dy\wedge \mu_1\wedge\cdots\wedge\mu_{q-1}) (Z\gamma,W\gamma, V_{q-1}\gamma, d_\t\gamma)
 \\&
  \sim\det \begin{bmatrix}
Zx &W x & V_{q-1}   x &d_\t x
\\
Zy & Wy & V_{q-1}y& d_\t y
\\
\mu_1 (Zt) & \mu_1 (Wt) & \mu_1 (V_{q-1}t) & \mu_1 (d_\t t)
  \\\vdots &\vdots &\vdots
 \\
  \mu_{q-1} (Zt) & \mu_{q-1} (W t) &   \mu_{q-1} (V_{q-1} t) &\mu_{q-1}  (d_\t t)
  \end{bmatrix}.
\end{aligned}
\end{equation*}
where $\sim$ means that the left-hand side vanishes if and only if the right-hand side vanishes. Here $x=[x_1,\dots, x_q]^T
$, while   $W$ stands for $[V_1,\dots, V_{q-2}]$ and $d_\t=[\p_{\t_1},\dots, \p_{\t_q}]$.

We already calculated the $Z$-derivatives in~\eqref{bottiglia} and~\eqref{bottiglia2}. Namely  $Zx(\frac{2\pi}{|\t|},\xi,\eta,\t)=0  $,  $Zy(\frac{2\pi}{|\t|},\xi,\eta,\t)=0$ and  $Z t(\frac{2\pi}{|\t|},\xi,\eta,\t)=-\frac{2\pi\eta}{|\t|}P_\t^\perp\xi$. By~\eqref{duale} we have then
$\mu_j (Zt)=0$, for $j\leq q-2$ and $\mu_{q-1}(Zt)=(P_\t^\perp\xi)^T
(-\frac{2\pi\eta}{|\t|}P_\t^\perp\xi)=-\frac{2\pi\eta}{|\t|}|P_\t^\perp\xi|^2\neq 0
$, by our assumptions.  Therefore, if $q\geq 3$,
\begin{equation}\label{radiofonico}
\begin{aligned}
\det M \sim &\det
 \begin{bmatrix}
 W  x &  V_{q-1}x & d_\t x
 \\
Wy  & V_{q-1}y & d_\t y
\\
\mu_1(W t)    &   \mu_1  (V_{q-1} t) &\mu_1  d_\t t
\\ \vdots &\vdots &\vdots
\\
\mu_{q-2}(W t)    &  \mu_{q-2}  (V_{q-1} t)
&  \mu_{q-2} (d_\t t)
\end{bmatrix}.
\end{aligned}
\end{equation}
If $q=2$, we have instead the simpler form
\begin{equation*}
\begin{aligned}
\det M \stackrel{q=2}{\sim}  & =
\det \begin{bmatrix}
 V_1  x      & d_\t x
 \\
V_1  y &    d_\t y
\end{bmatrix}.
\end{aligned}
\end{equation*}

Next, for $k\leq q-1$ we have $V_k x=(-\langle\omega_k,\xi\rangle+\eta D_{\omega_k})\frac{2\pi}{|\t|}P_\t^\perp\xi=\frac{2\pi\eta}{|\t|}P_\t^\perp\omega_k=\frac{2\pi\eta}{|\t|} \omega_k$, because $\omega_k\perp \t $ for all    $k\leq q-1$. Easily, $V_k y=0$ and finally
\begin{equation*}
\begin{aligned}
 V_k t &=  \frac{\pi}{|\t|^3}(-\langle\xi,\omega_k\rangle\p_\eta+\eta D_{\omega_k})\Big((|P_\t\xi|^2+\eta^2)
 \t-2P_\t^\perp\xi\xi^T\t\Big)
\\&=\frac{\pi}{|\t|^3}\Big(-2\eta\langle\xi,\omega_k\rangle\t+2\eta\langle P_\t\xi,  P_\t\omega_k\rangle \t  -
2\eta P_\t^\perp\omega_k \xi^T\t -2\eta P_\t^\perp\xi \omega_k^T\t\Big)
\\& =-\frac{2 \pi \eta}{|\t|^3}\Big\{ \langle\xi,\omega_k\rangle\t+
    \langle \xi ,\t\rangle\omega_k \Big\}
 \end{aligned}
\end{equation*}
where we used $\omega_k\in\t^\perp$.
Therefore, for $j,k\leq q-2$ we have $\mu_j(V_k t)=\omega_j^T (V_kt)= -\frac{2\pi\eta}{|\t|^3}\langle \xi,\t\rangle  \omega_j^T\omega_k$. Introducing the matrix
$\Omega=[\omega_1,\dots, \omega_{q-2}]\in\R^{q\times (q-2)}$, if $q\geq 3$, we can write the
first $q-2$ columns of the matrix in~\eqref{radiofonico} as
 \begin{equation*}
\begin{aligned}
  \begin{bmatrix}
\frac{2\pi \eta}{|\t|}  \Omega
\\
0
\\
-\frac{2\pi \eta}{|\t|^3}\langle \xi,\t\rangle \Omega^T\Omega
 \end{bmatrix}\in\R^{(2q-1)\times(q-2)}.
\end{aligned}
 \end{equation*}
 Passing to the $(q-1)$-th column, since we let $\omega_{q-1}=P_\t^\perp\xi$,   computations above give  $V_{q-1}x=\frac{2\pi\eta}{|\t|}P_\t^\perp\omega_{q-1}=
\frac{2\pi\eta}{|\t|}P_\t^\perp\xi $. Furthermore,  $V_{q-1}y=0$ and
\[
 V_{q-1}t=-\frac{2\pi\eta}{|\t|^3}\big( \langle\xi, P_\t^\perp\xi\rangle\t +\langle\xi, \t\rangle P_\t^\perp\xi\big)\quad\Rightarrow\quad \mu_j(V_{q-1}t)=0\quad\forall\;j\leq q-2.
\]
Ultimately, the $(q-1)$-th column becomes
$
 \begin{bsmallmatrix}
 |\t|^{-1}2\pi\eta P_\t^\perp\xi \\0\\0
 \end{bsmallmatrix}\in\R^{ 2q-1 }.
$

Let us pass to the columns involving   $\p_{\t_\a}$. In all differentiations below, we omit all terms that vanish when $s=2\pi/|\t|$. The calculations of $\p_{\t_\a}\gamma(s,\xi,\eta,\t)$ for $s\neq \frac{2\pi}{|\t|}$ would be much longer.
\begin{equation*}
\begin{aligned}
 \p_{\t_\a}x\Big|_{s=\frac{2\pi}{|\t|}}  &=\p_{\t_\a } \bigg(
 \tfrac{2}{|\t|} \sin\big(\tfrac{|\t|s}{2}\big)
 \Big\{
 P_\t\xi\,
\cos\big(\tfrac{|\t|s}{2}\big)
-\frac{ \t\eta^T}{|\t|} \sin\big(\tfrac{|\t|s}{2}\big)
\Big\}
  + s
P_{\tau}^\perp \xi\bigg)\Big|_{s=\frac{2\pi}{|\t|}}
\\&
=\frac{2}{|\t|}\cos(\tfrac{|\t|s}{2})\frac s2\frac{\t_
\a}{|\t|}(-P_\t\xi)+\frac{2\pi}{|\t|}\p_{\t_\a}\Big(\xi-\frac{\t\t^T}{|\t|^2}\xi\Big)\Big|_{s=\frac{2\pi}{|\t|}}
\\&
= \frac{2\pi}{|\t|^3}\t_\a P_\t\xi+\frac{2\pi}{|\t|}
\Big(\frac{2}{|\t|^3}\frac{\t_\a}{|\t|} \t\t^T\xi-\frac{e_\a\t^T}{|\t|^2}\xi-\frac{\t e_\a^T}{|\t|^2}\xi\Big)
\\&
=\frac{2\pi}{|\t|^3}\big( 3\t_\a P_\t\xi- e_\a\langle\xi,\t\rangle-\t\xi_\a\big).
\end{aligned}
\end{equation*}
So, the $q\times q$ north east block is
$
 d_\t x |_{s=\frac{2\pi}{|\t|}}=\frac{2\pi}{|\t|^3}\big( 3\langle\xi,\t\rangle P_\t -\langle\xi,\t\rangle I_q-\t\xi^T\big)\in\R^{q\times q}
$.
An analogous computation furnishes $ \p_{\t_\a}y|_{s=2\pi/|\t|}=\frac{2\pi}{|\t|^3}\eta\t_\a$ for $\a=1,\dots, q$.
To conclude, for $q\geq 3$, we have to  calculate derivatives $\p_{\t_\a} t|_{s=2\pi/|\t|}$.  Keeping into account that  we need only to know    $\mu_j (\p_{\t_\a}t)=\omega_j^T \p_{\t_\a}t$ with $j\leq q-2$, since $\Span\{\omega_1,\dots, \omega_{q-2}\}=\{\t,P_\t^\perp\xi\}^\perp $, below we  work ignoring terms in $\Span\{\t,P_\t^\perp\xi\}$ writing below  $u\simeq u' $ when $u-u'\in\Span\{\t,P_\t^\perp\xi\}\subset\R^q$.
\begin{equation*}
\begin{aligned}
 \p_{\t_\a} t|_{s=2\pi/|\t|}&
 =
\p_{\t_\a}
 \bigg(  s^2 U\big(\tfrac{|\t|s}{2} \big)
\big\{|P_\t\xi|^2 +|\eta^2| \big\}\frac{\t}{|\t|}
\\& +s^2 V\big(\tfrac{|\t|s}{2} \big)
\Big\{- P_\t^\perp\xi \,\eta  \sin\big(\tfrac{|\t|s}{2}\big)
+P_\t^\perp\xi\frac{\xi^T\t}{|\t|}   \;
\cos \big(\tfrac{|\t|s}{2} \big)
\Big\}\bigg)
\\&
\simeq\Big(\frac{2\pi}{|\t|}\Big)^2 U(\pi)(|P_\t\xi|^2+\eta^2)\frac{1}{|\t|}\p_{\t_\a}\t
+ \Big(\frac{2\pi}{|\t|}\Big)^2 V(\pi) \p_{\t_\a}P_\t^\perp\xi\Big(\frac{\xi^T\t}{|\t|}\cos\pi\Big)
\\&=\frac{\pi}{|\t|^2}  (|P_\t\xi|^2+\eta^2)\frac{e_\a}{|\t|}
+\frac{2\pi}{|t|^2} \frac{e_\a\t^T \xi}{|\t|^2}\frac{\xi^T\t}{|\t|}
 =\frac{\pi}{|\t|^3}\big(\eta^2+3|P_\t\xi|^2\big) e_\a .
\end{aligned}
\end{equation*}
The south east $(q-2)\times q$-block has elements $(\mu_j(\p_{\t_\a} t))=\omega_j^T(d_{\t_\a} t))$, with
$j\leq q-2$ and $\a=1,\dots, q$. Ultimately,  the matrix in~\eqref{radiofonico} takes the form
\begin{equation*}
  M= \begin{bmatrix}
\frac{2\pi\eta}{|\t|} \Omega & \frac{2\pi \eta}{|\t|}P_\t^\perp\xi &
  \frac{ 2\pi}{|\t|^3}   \big( 3\langle\xi,\t\rangle P_\t -\langle\xi,\t\rangle I_q-  \t \xi^T\Big)
\\
0 & 0 &\frac{2\pi \eta}{|\t|^3}\t^T
\\
-\frac{2\pi \eta}{|\t|^3}\langle\xi,\t\rangle \Omega^T\Omega & 0
& \frac{\pi}{|\t|^3}\big(\eta^2+3|P_\t\xi|^2\big) \Omega^T
 \end{bmatrix}
\end{equation*}
where $\Omega=[\omega_1,\dots, \omega_{q-2}]\in\R^{q\times (q-2)}$. By linear algebra, \footnote{Note that the north east-term containing the factor  $P_\t=\frac{\t}{|\t|^2}\t^T$ can be eliminated by subtracting to each of its rows a suitable multiple of the $(q+1)$-th row.  All other simplifications are multiplications of some row/columns for nonzero scalars.}
$ M$ has full rank if and only if
\[
\wh M:= \begin{bmatrix}
 \Omega &  P_\t^\perp\xi &
-2 \big(  \langle\xi,\t\rangle I_q+  \t \xi^T\Big)
\\
0 & 0 & \t^T
\\
-\frac{ \langle\xi,\t\rangle }{|\t|^2} \Omega^T\Omega & 0
&  \big(\eta^2+3|P_\t\xi|^2\big) \Omega^T
 \end{bmatrix}\in\R^{(2q-1)\times(2q-1)}
\]
has full rank.
If $q=2$, we have the simpler matrix
\begin{equation}\label{facile}
\wh M\stackrel{q=2}{=}
 \begin{bmatrix}
    P_\t^\perp\xi &
-2 \big(  \langle\xi,\t\rangle I_2+  \t \xi^T\Big)
\\
  0 & \t^T
 \end{bmatrix}\in\R^{3\times 3}.
\end{equation}

To conclude the proof we prove the following claim.

\step{Claim.} The matrix $\wh M$ has  trivial kernel.

Proof of the claim for $q\geq 3$ and    $\langle \xi,\t\rangle=0$. We show that $\wt M$ has trivial kernel. Let     $a\in\R^{q-2}$, $b\in\R$ and $c\in\R^q$. If $\langle\xi,\t\rangle=0$, the system $\wh M\begin{bsmallmatrix}
                                 a\\b\\c
                                \end{bsmallmatrix}=0$ becomes
\begin{equation}\label{topo}
 \left\{ \begin{aligned}
&\Omega a+P_\t^\perp\xi b -2\t\langle\xi,c\rangle=0
\\& \langle \t,c\rangle=0
\\& \big(\eta^2+3|P_\t\xi|^2\big) \Omega^T c=0.
\end{aligned}
\right.
\end{equation}
By~\eqref{duale},   $\Omega a\in\Span\{\t,P_\t^\perp\xi\}$. Then, the first line gives three separate conditions and we get
\begin{equation}\label{gero}
 \left\{ \begin{aligned}
&\Omega a=0,\quad   b=0,\quad  \langle\xi,c\rangle=0
\\& \langle \t,c\rangle=0
\\& \big(\eta^2+3|P_\t\xi|^2\big) \Omega^T c=0
\end{aligned}
\right.
\end{equation}
Furthermore, since $\Omega=[\omega_1,\dots, \omega_{q-2}]$ has no kernel, we get  $a=0$. Concerning $c$, we see that  it is orthogonal both  to $\t$ and to $\Span\{\omega_1,\dots, \omega_{q-2}\}$, by the third line of~\eqref{gero}. Then, $c=\lambda P_\t^\perp \xi$ for suitable $\lambda\in\R$. Again from the first line, we get
$0=\langle\xi,c\rangle=\langle \xi,\lambda P_\t^\perp\xi\rangle =\lambda|P_\t^\perp\xi|^2 $,  which implies $c=0$ as we wished.

Proof of the claim, $q\geq 3$ and $\langle\xi,\t\rangle\neq 0$.  The system becomes
\begin{equation*}
 \left\{ \begin{aligned}
&\Omega a+P_\t^\perp\xi b  -2\langle\xi,\t\rangle c -2\t\langle\xi,c\rangle=0
\\& \langle \t,c\rangle=0
\\& -\frac{\langle\xi,\t\rangle}{|\t|^2}\Omega^T\Omega a +\big(\eta^2+3|P_\t\xi|^2\big) \Omega^T c=0.
\end{aligned}
\right.
\end{equation*}
Multiply from the left  the first line of~\eqref{topo} by $\frac{\langle\xi,\t\rangle}{|\t|^2}\Omega^T$ and add to the third. After elementary simplifications based on property $\Omega^T \t=\Omega^T P_\t^\perp\xi=0$, this gives
$ (\eta^2+ |P_\t\xi|^2  )\Omega^Tc=0.
$
Thus, $\Omega^Tc=0$ which implies $c\in\Span \{\t,P_\t^\perp\xi\}$. Arguing as above, we write   $c=\lambda P_\t^\perp\xi$ and   projecting orthogonally the first line along $\t$, we get  $c=0$. The third line gives  also $\Omega^T\Omega a=0$, which implies $a=0$ (independence of $\omega_1,\dots, \omega_{q-2}$ in $\R^q$  implies that $\Omega^T\Omega$ is nonsingular).  The first line gives now  $b=0$ and the proof is finished.

Proof of the claim, $q=2$. The system $\wh M \begin{bsmallmatrix}
                                                        b\\c
                                                       \end{bsmallmatrix}
$ with $b\in\R$ and $c\in\R^2$ becomes
\begin{equation*}
 \left\{
 \begin{aligned}
 &P_\t^\perp\xi b-2\langle\xi,\t\rangle c-2\t\langle\xi,c\rangle=0
 \\
 & \t^T c=0.
        \end{aligned}
\right.
\end{equation*}
Here $\R^2=\Span\{\t,P_\t^\perp\xi\}$ and the second line gives $c=\lambda P_\t^\perp\xi$ for some $\lambda$. Then the first becomes $P_\t^\perp\xi(b-2\langle\xi,\t\rangle c)-2\t\lambda|P_\t^\perp\xi|^2=0$, which by independence of $\t$ and $P_\t^\perp\xi$ provides  $\lambda=b=0$. This concludes the proof of the Theorem~\ref{taglia}.
\end{proof}

We are ready to prove Theorem~\ref{loggia}.

\begin{proof}
 Let $q\geq 2$ and $p=1$ and let $(x,0,t)\in\Cut(\G_{q1})$. Write $(x,0,t)=\gamma(2\pi,\xi,\eta, \t)$ with $|\t|=1$ and $\eta^2+|P_\t^\perp\xi|^2\neq 0$.    We know by Proposition~\ref{piuno} and Theorem~\ref{taglia}  that  this point is conjugate if and only if $\eta|P_\t^\perp\xi|=0$.  Condition  $P_\t^\perp\xi=0$  is equivalent to  $x=0$, by the first line of~\eqref{ovest}.  If $x\neq 0$, then $\ker x$ is trivial, and Remark~\ref{spiego} holds.  In particular, by~\eqref{spiegazione}, condition  $\eta =0$ holds   if and only if
$|P_{\Im x}^\perp t|-\pi|x^\dag t|^2=0$, as required. \color{black} Note that here we have $|x^\dag t|^2=\frac{\langle x,t\rangle^2}{|x|^4}$.
\end{proof}

\appendix
\section{Appendix}

The following  lemma on quadratic Hamiltonian systems has been used in the proof of Theorem~\ref{taglia}. We include it for completeness. This lemma is general and we use then standard Hamiltonian notation.
\begin{lemma}\label{liouville}
Let $(x,p)\in\R^N\times\R^N\simeq T^*(\R^N)$ and consider the Hamiltonian
\begin{equation}\label{amilto}
H(x,p)=\frac 12   \langle M(x)p , p\rangle
\end{equation}
where  $M (x)\in\R^{N\times N}$ is symmetric, positive semidefinite and depends smoothly on $x$.
Denote by    $\Lambda=\{ p\in\R^n =T^*_{0}\R^N : H(0,p)=\frac 12\}$.  Given $p\in\Lambda $,  denote by $t\in I \mapsto (X(t,p), P(t,p))$  the solution of the Hamiltonian system with initial data $(X(0,p), P(0,p))=(0,p)$, where $I\subset \R$ is an interval containing  $0$.   Let $\sigma\in\R\mapsto p(\s) \in \Lambda$ be a smooth  path. Then we have
\begin{equation*}\sum_j P_j(t, p(\s))\frac{\p  }{\p\s } X_j(t, p(\s))=
\sum_j P_j(0, p(\s))\frac{\p  }{\p\s } X_j(0, p(\s))=0\quad\forall\;t,\s.\label{tengential}
\end{equation*}
In other words, the Liouville form $\ell :=\sum_j p_j dx_j$ satisfies $\ell  (\frac{\p}{\p\s}X(t, p_\s))=0$, for all $t,\s$.
\end{lemma}

\begin{proof}
Write briefly $X_j=X_j(t, p(\s))$ and $P_j=P_j(t, p(\s))$ and omit summation on index $j\in\{1,\dots,N\}$. We must see that the following is zero:
\[
(*):=\p_t(P_j \p_\s X_j )=\p_t P_j \p_\s X_j+ P_j\p_t\p_\s X_j
 \]
Since the Hamiltonian is constant along flow, we get
$
H(X , P )=   H(0,p(\s))=\frac 12
$ identically in $t,\s$.
Differentiating with respect to $\sigma$ gives
\[
0=\p_{x_j}H(X ,P )\p_\s X_j+\p_{p_j}H(X ,P )\p_\s P_j=-\p_tP_j \p_\s X_j+ \p_t X_j\p_\s P_j.
\]
Thus we may rewrite $(*)$ as
\[\begin{aligned}(*)&= \p_t X_j\p_\s P_j+ P_j\p_t\p_\s X_j= \p_t X_j\p_\s P_j+ P_j\p_\s\p_t X_j
=\p_\s (P_j\p_t X_j)=\p_\s (P_j\p_{p_j}H)
\\&\stackrel{(\dag)}{=} \p_\s H(X ,P)   = \p_\s  H(0, p(\s))=0.
\end{aligned}
\]
In $\stackrel{(\dag)}{=}$  we used  the form~\eqref{amilto} of the Hamiltonian.
\end{proof}

\section*{Acknowledgements}
 The authors are   supported by PRIN 2022 F4F2LH - CUP J53D23003760006 ``Regularity problems in sub-Riemannian structures''.

 The authors are also members of the {\it Gruppo Nazionale per
l'Analisi Matematica, la Probabilit\`a e le loro Applicazioni} (GNAMPA)
of the {\it Istituto Nazionale di Alta Matematica} (INdAM).

 The auhors thank Ye Zhang, for a careful reading of the first version of the manuscript and for suggesting some improvements.


\footnotesize




\def\cprime{$'$} \def\cprime{$'$}
\providecommand{\bysame}{\leavevmode\hbox to3em{\hrulefill}\thinspace}
\providecommand{\MR}{\relax\ifhmode\unskip\space\fi MR }
\providecommand{\MRhref}[2]{%
  \href{http://www.ams.org/mathscinet-getitem?mr=#1}{#2}
}
\providecommand{\href}[2]{#2}

\normalsize

\end{document}